\documentclass{amsart}[12pt]

\usepackage{amsfonts,amssymb,amsmath,amsthm}
\usepackage{amsaddr}

\usepackage{graphicx}
\usepackage{epstopdf}
\usepackage{caption}
\usepackage{subcaption}
\usepackage{color}

\usepackage[pdfpagelabels]{hyperref}

\newtheorem{thm}{Theorem}[section]
\newtheorem{prop}{Proposition}[section]
\newtheorem{lem}{Lemma}[section]
\newtheorem{cor}{Corollary}[section]

\theoremstyle{definition}
\newtheorem{defi}{Definition}[section]
\newtheorem{nota}{Notation}[section]

\theoremstyle{remark}
\newtheorem{rem}{Remark}[section]

\newcommand{\eea}{\end{eqnarray}}

\def\lb{\label}

\newcommand{\rr}{\mathbb{R}}
\newcommand{\cc}{\mathbb{C}}
\newcommand{\Sp}{\mathrm{Sp}}
\newcommand{\mf}{\mathfrak{M}}
\newcommand{\mr}{\mathfrak{R}}

\newcommand{\cf}{\mathcal{F}}

\newcommand{\al}{\alpha}

\newcommand{\ga}{\gamma}
\newcommand{\gm}{\gamma}
\newcommand{\tht}{\theta}
\newcommand{\lmd}{\lambda}
\newcommand{\zt}{\zeta}
\newcommand{\sg}{\sigma}
\newcommand{\om}{\omega}
\newcommand{\kp}{\kappa}

\newcommand{\vep}{\varepsilon}

\newcommand{\vr}{\varrho}

\newcommand{\gh}{\hat{\gamma}}

\newcommand{\ey}{\frac{1}{2}}

\newcommand{\xd}{\dot{x}}
\newcommand{\rd}{\dot{r}}
\newcommand{\thd}{\dot{\tht}}
\newcommand{\xb}{\bar{x}}
\newcommand{\rb}{\bar{r}}

\newcommand{\bh}{\hat{B}}
\newcommand{\bhy}{\hat{B}^{(1)}}
\newcommand{\bhe}{\hat{B}^{(2)}}

\newcommand{\eh}{\hat{e}}
\newcommand{\lh}{\hat{\lmd}}

\newcommand{\wg}{\wedge}

\newcommand{\e}{\mathbf{e}}

\newcommand{\uh}{\mathfrak{U}}

\begin{document}

\title[Index Theory for Parabolic Solutions]{An Index Theory for Zero Energy Solutions of the Planar Anisotropic Kepler Problem}

\author{Xijun Hu}
\address{Department of Mathematics, Shandong University, P.R. China}
\email{xjhu@sdu.edu.cn}

\author{Guowei Yu}
\address{Dipartimento di Matematica ``Giuseppe Peano'', Universit\`a degli Studi di Torino, Italy}
\email{guowei.yu@unito.it}

\thanks{Both authors thank the support of NSFC(No.11425105). The second author also acknowledges the support of Fondation Sciences Math\'ematiques de Paris and the ERC Advanced Grant 2013 No.  339958 ``Complex Patterns for Strongly Interacting Dynamical Systems - COMPAT''}.

  
\date{}

\begin{abstract}
In the variational study of singular Lagrange systems, the zero energy solutions play an important role. In this paper we find a simple way of computing the Morse indices of these solutions for the planar anisotropic Kepler problem. In particular an interesting connection between the Morse indices and the oscillating behaviors of these solutions discovered by the physicist M. Gutzwiller is established.      
\end{abstract}

\maketitle

\bigskip


\bigskip


\section{Introduction}

Lagrangian systems with singular potentials have been studied by many authors due to their connection with celestial mechanics and relevant problems in physics, see \cite{BR89}, \cite{BR91}, \cite{AZ93}, \cite{AZ94}, \cite{Tn93a} and the references within. In this paper, we study the $2$-dimension singular Lagrangian system
\begin{equation} \ddot{x}(t)= \nabla U(x(t)), \;\; x(t)=(x_1(t), x_2(t)) \in \rr^2,  \label{eq: Lag equation}  \end{equation}
with $U$ being a positive, $(-\al)$-homogeneous potential for some $0< \al <2$, i.e.
\begin{equation}
\label{eq: U}  U(x) = \frac{\uh(x/|x|)}{|x|^{\al}}, \;\; \text{where } \uh \in C^2(\mathbb{S}^1, (0, +\infty)) \text{ and } \mathbb{S}^1 := [-\pi, \pi]/\{\pm \pi\}.
\end{equation}

This can be seen as a generalization of the planar anisotropic Kepler problem, introduced by physicist Gutzwiller (\cite{Gz73}, \cite{Gz77}) and further studied by Devaney (\cite{Dv78}, \cite{Dv81}), where
\begin{equation}
\label{eq: uh anisotropic kepler} U(x) = \frac{1}{\sqrt{\mu x_1^2 + x_2^2}}, \; \text{ for some } \mu >1.
\end{equation}
Such a potential describes the motion of an electron in a semiconductor by an impurity of the donor type and reveals the connection between chaotic behaviors in classic and quantum mechanics. The general case we are considering also applies to the Kepler problem and the isosceles three body problem, see Section \ref{sec: application}.

Solutions of \eqref{eq: Lag equation} are critical points of the action functional 
\begin{equation} \mathcal{F}(x; t_1, t_2):=\int_{t_1}^{t_2} L(x(t), \xd(t)) dt, \;\; x \in W^{1,2}([t_0,t_1],\mathbb{R}^2 \setminus \{0 \} )\end{equation}
where the Lagrangian
$$ L(x, \dot{x})= K(\xd)+ U(x) = \ey|\xd|^2 +U(x). $$
The corresponding Hamiltonian $H(\xd(t), x(t)) =K(\xd(t))-U(x(t))$ represents the total energy and is a constant along a solution. Under polar coordinates 
$$ x=(x_1, x_2)=(r\cos \theta, r\sin\theta), \;\;  (r, \tht) \in [0, +\infty) \times \mathbb{S}^1, $$
$\uh$ depends only on $\tht$ and the Lagrangian becomes
\begin{equation}
\label{eq: Lagrangian}  L(r, \tht, \dot{r}, \dot{\tht})= K+U=\frac{1}{2}(\dot{r}^2+r^2\dot{\theta}^2)+ \frac{\uh(\theta)}{r^{\al}}. 
\end{equation}

Let $x(t) \in \rr^2 \setminus \{0\}, t \in (T^-, T^+) \subset \rr \cup \{\pm\infty\}$, be a solution of \eqref{eq: Lag equation}, we are mainly interested in the following three types of solutions. 
\begin{defi}
\label{def: parabolic}
$x(t)$ will be called a \textbf{parabolic} solution, if 
\begin{enumerate}
\item[(i).] $T^{\pm} =\pm \infty$, $\lim_{t \to T^{\pm}}|x(t)| = +\infty$ and $\lim_{t \to T^{\pm}}|\xd(t)|=0$, 
\end{enumerate}
a \textbf{collision-parabolic} solution, if 
\begin{enumerate}
\item[(ii).] $T^- \in \rr$ and $x(T^-)=\lim_{t \to T^-}x(t)=0$;
\item[(iii).] $T^+= +\infty$, $\lim_{t \to T^+} |x(t)| = +\infty$ and $\lim_{t \to T^+}|\xd(t)| = 0$, 
\end{enumerate}
and a \textbf{parabolic-collision} solution, if 
\begin{enumerate}
\item[(iv).] $T^+ \in \rr$ and $x(T^+)= \lim_{t \to T^+}x(t)=0$;
\item[(v).] $T^-=-\infty$, $\lim_{t \to T^-} |x(t)| = +\infty$ and $\lim_{t \to T^-}|\xd(t)| = 0$.
\end{enumerate}
If $\tht(t) \equiv \text{Constant}$, $ \forall t \in (T^-, T^+)$, we also call $x(t)$  a \textbf{homothetic} solution. 
\end{defi}
Notice that a parabolic solution is always non-homothetic, as a homothetic solution must collide with the origin at a finite time in the future or past. Clearly all the solutions introduced in Definition \ref{def: parabolic} must have zero energy. Meanwhile the reverse statement is also true under some non-degenerate condition. 

\begin{thm} [Devaney \cite{Dv81}]
\label{thm: zero energy} If the critical points of $\uh$ are isolated in $\mathbb{S}^1$, then each zero energy solution $x(t)$, $t \in (T^-, T^+)$, must be one of the three types of solutions defined in Definition \ref{def: parabolic}. Furthermore in polar coordinates $x(t)=(r\cos \tht, r \sin \tht)(t)$, $\tht(t)$ converges to some critical points of $\uh$, as $t$ goes to $T^{\pm}$.
\end{thm}

We find these zero energy solutions interesting due to the following reasons: first under McGehee coordinates (see \cite{Mg74}, \cite{Dv81}, \cite{Mk81} or Section \ref{sec: McGehee coordinates}), their projections on the collision manifold (obtained after blowing up the singularity at the origin) become equilibria and heteroclinic orbits between these equilibria, which means they may be used to build up complex trajectories, see \cite{Mk89} and \cite{MM15}; second, in \cite{BTV}, \cite{BTV14} and \cite{LM14}, the existence/absence of parabolic solutions are shown to be connected with the absence/existence of collision in the action minimizers of the Bolza problem (fixed-end); third, in the variational study of the singular Lagrange systems, they are usually what one gets after the \emph{blow-up} argument(\cite{Tn93a}, \cite{FT04}) and play a key role in proving the absence of collision in the corresponding critical points.

The main novelty of our paper is to study these solutions from an index theory point of view. To be precise, given a zero energy solution $x(t)$, $t \in (T^-, T^+)$, we define its \textbf{Morse index} as 
\begin{equation}
\label{dfn: morse index}  m^-(x)=\lim_{n \to \infty}m^-(x; t^-_n, t^+_n). 
\end{equation}
where $T^-< t^-_n < t^+_n < T^+$ satisfies $\lim_{n \to +\infty} t^{\pm}_n= T^{\pm}$. For any $t_1<t_2$, $m^-(x; t_1, t_2)$ is the dimension of the largest subspace of $W_0^{1,2}([t_1, t_2], \rr^2 \setminus \{0\})$, where the second derivative $d^2\mathcal{F}(x; t_1, t_2) <0$. By the monotone property in \cite{CH},  
\begin{equation} m^-(x; t_1,t_2)\leq m^-(x; t^*_1,t^*_2), \; \text{ if } \; t^*_1\leq t_1, t_2\leq t^*_2. \lb{4.1} \end{equation}
Hence $m^-(x)$ is well defined and independent of the choice of $t^{\pm}_n$.

The computation of Morse index is not an easy job, especially along the directions that are not orthogonal to the solution. Our result gives a simple way of computing the Morse index of a zero energy solution, and quite interestingly it is connected with the oscillating behavior of the solution discovered numerically by Gutzwiller and proven analytically by Devaney: 

{\em Let $\al=1$ and $\uh(\tht)=(\mu \cos^2 \tht + \sin^2 \tht)^{-\frac{1}{2}}$ with $\mu >1$, then $\{-\pi/2, 0, \pi/2, \pi\}$ are the critical points of $\uh$. If $x(t)=(r\cos \tht, r \sin \tht)(t)$ is a collision solution of \eqref{eq: Lag equation} with $x(0)=0$ (not necessarily with zero energy), as $t \to 0$, $\tht(t)$ converges to one of the critical point. When the the critical point belongs to $\{\pm \pi/2\}$, then when $\mu > 9/8$, the corresponding trajectory in $\rr^2$ oscillates along the vertical axis $\{x_1 \equiv 0\}$, as it approaches to the origin; meanwhile when the critical point belongs to $\in \{0, \pi\}$, then such oscillating behavior does not exist along the horizontal axis $\{x_2 \equiv 0\}$. See Figure \ref{fig: osc} for corresponding numerically simulations, where the corresponding graphs of the function $\tht(\tau)$ are given ($\tau$ is a new time parameter that will be given later). This may also be seen from the phase portrait given in Figure \ref{fig: phase}.}

\begin{figure} 
    \centering
     \begin{subfigure}[b]{0.48\textwidth}
      \includegraphics[width=\textwidth]{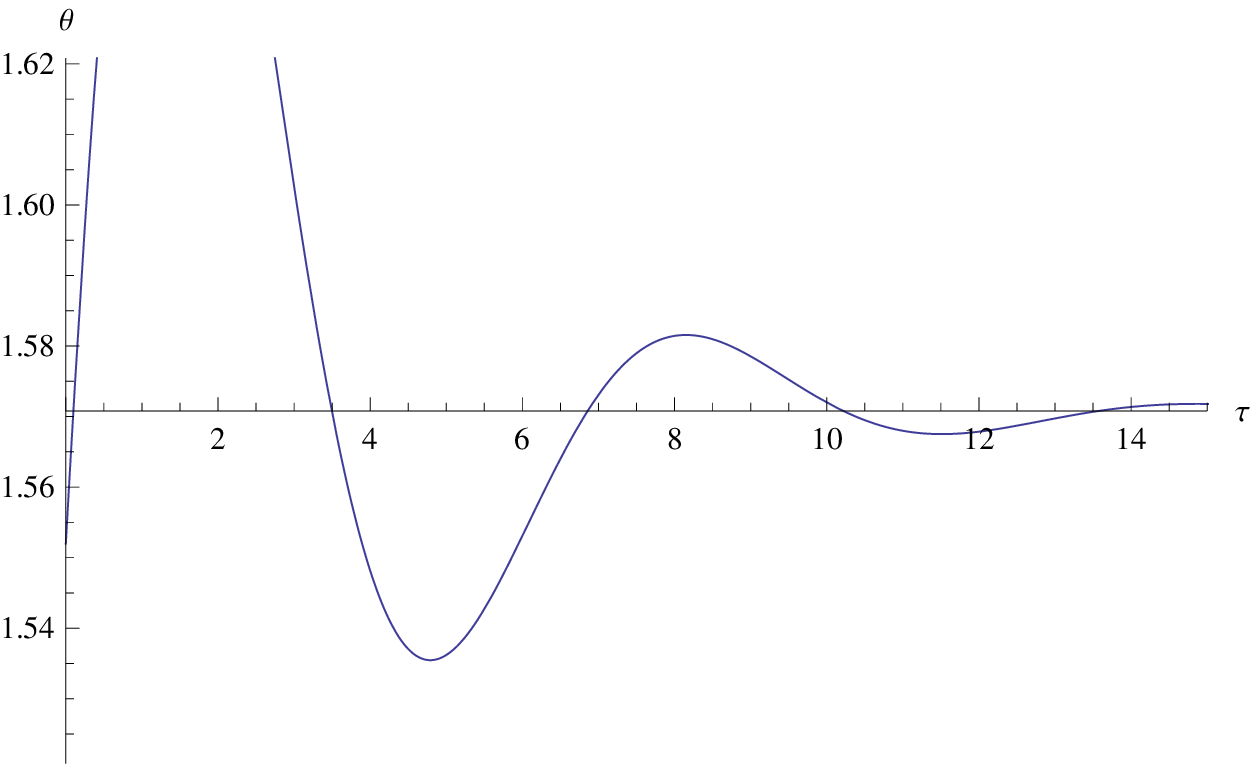}
        \caption{$\tht(t) \to \pi/2$}
    \end{subfigure}
    ~
    \begin{subfigure}[b]{0.48\textwidth}
        \includegraphics[width=\textwidth]{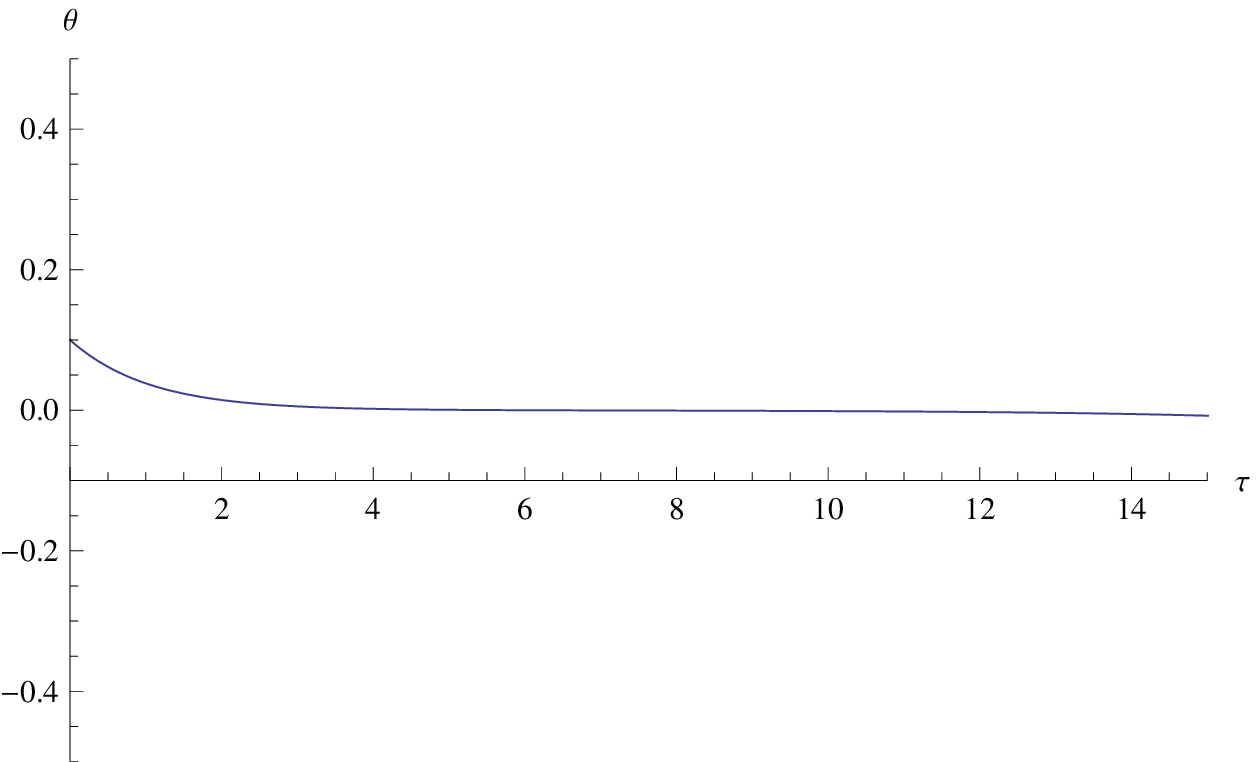}
        \caption{$\tht(t) \to 0$}
    \end{subfigure}
    \caption{}
    \label{fig: osc}
\end{figure}

Inspired by the above phenomena, we call $i(x)$ the \textbf{oscillation index} of $x(t)$:
\begin{equation} \label{dfn: osc index} i(x):= \begin{cases}
{\#}\{t \in (T^-, T^+)| \; \dot{\theta}(t)=0\}, \; & \text{if } x(t)  \text{ is non-homothetic}, \\
0, \; & \text{if } x(t) \text{ is homothetic}. 
\end{cases}
\end{equation}  
\begin{rem}
If $x(t)$ is homothetic, $\dot{\tht}(t)\equiv 0$, $\forall t$, so there is no oscillation at all. Meanwhile if $x(t)$ is non-homothetic, by Remark \ref{rmk:isolated time}, $\{t \in (T^-, T^+)| \; \dot{\theta}(t)=0\}$ is isolated in $(T^-, T^+)$.
\end{rem}

\begin{thm}\label{thm: osc morse} Let $x(t)= (r \cos \tht, r \sin \tht)(t)$, $t \in (T^-, T^+)$, be a non-homothetic zero energy solution of \eqref{eq: Lag equation} with $\lim_{t \to T^{\pm}} \tht(t) = \tht^{\pm}_0$. Then $\tht_0^{\pm}$ are critical points of $\uh$. Moreover when both of them are are non-degenerate, i.e. $\uh_{\tht \tht}(\tht_0^{\pm})\ne 0$, then 
\begin{enumerate}
\item[(a).] if at least one of $\Delta(\tht_0^{\pm})$ is negative, then $m^-(x)=i(x)= +\infty$, where 
\begin{equation}
\label{eq: delta} \Delta(\tht) := \frac{(2-\al)^2}{2}\uh(\tht) +4 \uh_{\tht \tht}(\tht), \;\; \forall \tht \in \mathbb{S}^1,
\end{equation}
\item[(b).] if both $\Delta(\tht_0^{\pm})$ are positive, then $ m^-(x)-i(x) = 0 \text{ or } 1,$ and in particular, $m^-(x)-i(x)=0$, when $\tht_0^-$ is a local minimizer of $\uh$. 
\end{enumerate}
\end{thm}

Because of degeneracy, Theorem \ref{thm: osc morse} does not hold for homothetic solutions. Instead we have the following result. 

\begin{thm}
\label{thm: homothtic} Let $\bar{x}(t)=\rb (t)(\cos \tht_0, \sin \tht_0)$ be a homothetic zero energy solution of \eqref{eq: Lag equation}, where $\tht_0$ is a critical point of $\uh$, then
\begin{enumerate}
\item[(a).] if $\Delta(\tht_0) < 0$, $m^-(\bar{x})= +\infty$,
\item[(b).] if $\Delta(\tht_0) \ge 0$, $m^-(\bar{x})=0$.
\end{enumerate}
\end{thm}

\begin{rem}
\begin{enumerate}
\item Each critical point $\tht_0$ of $\uh$ corresponds to two equilibria on the collision manifold and the sign of $\Delta(\tht_0)$ is related to the spectra of the linearized vector field at those equilibria: when $\Delta(\tht_0)<0$, $(\psi_0, \theta_0)$ is a stable $($or unstable$)$ \textbf{focus} with the nearby orbits asymptotically spiral into $($or away from$)$ $(\psi_0, \tht_0)$  $($see Section \ref{sec: McGehee coordinates}$)$. 
\item When $\Delta(\tht_0)<0$, the Morse index of a collision solution of the $N$-body problem was first investigated in \cite{BS}, where results similar to property (a) in both Theorem \ref{thm: osc morse} and \ref{thm: homothtic} were obtained.
\item Recently in \cite{BHPT} the Morse indices of both collision and complete parabolic solutions of the $N$-body problem are studied in more details. In particular the case with $\Delta(\tht_0)>0$ $($called [BS]-condition$)$ is also considered there. 
\end{enumerate}
\end{rem}

Although we require the corresponding critical points of $\uh$ to be non-degenerate in Theorem \ref{thm: osc morse}, our approach may still work even when they are not. This is important as the $N$-body problem is highly degenerate due to symmetries. As a example, the Kepler-type problem with $\uh(\tht)$ being a constant, will be considered in Section \ref{sec: Kepler}.

Theorem \ref{thm: osc morse} has the following corollary(for a proof see Section \ref{sec: Morse Maslov indices}). A related result has been obtained recently in \cite{MMM17} for the planar three body problem.
\begin{cor}\label{cor: saddle to saddle} Following the notations from Theorem \ref{thm: osc morse}, if $x(t)$ be a parabolic solution with $\tht(t)$ converges to two non-degenerate global minimizers of $\uh$, then $m^-(x)= i(x)=0$. 
\end{cor}

The existence of parabolic solutions connecting two non-degenerate global minimizers of $\uh$ have been studied for the anisotropic Kepler problem with two degrees of freedom in \cite{BTV} and arbitrary finite degrees of freedom in \cite{BTV14}, where they are found as collision-free minimizers in the entire domain of time (under additional topological constraints in \cite{BTV}), so naturally their Morse index must be zero. Corollary \ref{cor: saddle to saddle} can be seen as a complementation of their results, as it says any parabolic solution connecting two global minimizers of $\uh$ must have zero Morse index.

We believe our result could be useful in deepening the variational study of the singular Lagrange systems including the classic $N$-body problem. In recent years, many new periodic and quasi-periodic solutions have been found as collision-free minimizers in the $N$-body problem under symmetric and/or topological constraints (see \cite{CM00}, \cite{FT04}, \cite{Ch08}, \cite{Y17}). However no result is available through minimax methods due to the problem of collision. Results from \cite{Tn93a}, \cite{BDT17} and \cite{Y17b} show that the Morse indices of zero energy solutions could be used to rule collisions in minimax approaches.

Our paper is organized as follows: Section \ref{sec: McGehee coordinates} contains a brief introduction of the McGehee coordinates; Section \ref{sec: asymptotic linear} gives the asymptotic analysis of the linear system along non-homothetic zero energy solutions, as they approach to the collision or infinity; Section \ref{sec: Morse Maslov indices}, studies the relations between various indices and contains proofs of our main results; Section \ref{sec: application} contains some applications of our results in celestial mechanics; Section \ref{sec: appendix B} gives a brief introduction of the Maslov index.

\section{McGehee coordinates and dynamics on the collision manifold}\label{sec: McGehee coordinates}

This section is an introduction to McGehee coordinates \cite{Mg74}. The results are not new and essentially due to Devaney (\cite{Dv78} and \cite{Dv81}). Their proofs either can be found in the above references or follow from direct computations, so will be omitted. 

The Hamiltonian corresponds to $L(r, \tht, \rd, \thd)$ given in \eqref{eq: Lagrangian} is 
\begin{equation}
\label{eq: Hamiltonian function} H(p_1, p_2, r, \tht)=\frac{1}{2}p_1^2+\frac{1}{2}r^{-2}p_2^2-r^{-\alpha}\uh(\theta), \; \text{ where } \; p_1=\dot{r}, p_2=r^2\dot{\theta}.
\end{equation}
Let $z=(p_1,p_2,r,\theta)^T$, the corresponding Hamiltonian system of \eqref{eq: Lag equation} is
\begin{equation} \label{eq: Hamiltonian equation}
\dot{z}=J \nabla H(z),  \; \text{ where }  J=\left( \begin{array}{cccc}0_2& -I_2 \\
I_2 & 0_2 \end{array}\right), \end{equation}
and
$$ \nabla H(z)=(p_1, r^{-2}p_2, \alpha r^{-1-\alpha}\uh(\theta)-r^{-3}p_2^2, -r^{-\alpha} \uh_{\tht}(\tht))^T . $$
  
Under the McGehee coordinates $(v, u, r, \tht)$
\begin{equation}
\label{eq: v u McGehee}  v=r^{\al/2} p_1=r^{\frac{\alpha}{2}}\dot{r},\;\; u=r^{-1+\frac{\al}{2}}p_2 =r^{1+\frac{\alpha}{2}}\dot{\theta}, 
\end{equation}
and  the new time parameter $\tau$ given by $dt = r^{1+\frac{\al}{2}} \, d\tau$, equation \eqref{eq: Hamiltonian equation} becomes
\begin{equation}
\begin{cases}
v'&=\frac{\alpha}{2}v^2+u^2-\alpha \uh(\theta), \\ 
u'&=(\frac{\alpha}{2}-1)uv+\uh_{\tht}(\theta),\\ 
r'&=rv,\\ 
\theta'&=u, 
\end{cases} \label{eq: McGehee 2} 
\end{equation}
where $'$ means $\frac{d}{d\tau}$ throughout the paper. 

The vector field now is well-defined on the singular set $\mf:=\{(v, u, r, \tht): r=0\}$. Moreover it is an invariant sub-manifold of \eqref{eq: McGehee 2}, which will be called the \emph{collision manifold}. In McGehee coordinates, the energy identity reads
\begin{equation}  \frac{1}{2}(u^2+v^2)-\uh(\theta)=r^{\alpha}H. \label{eq: energy}      \end{equation}
As a result whenever $r=0$ or $H=0$, 
\begin{equation}
\label{eq: energy 0} u^2+v^2=2\uh(\theta). 
\end{equation}
Plug this into the first equation of \eqref{eq: McGehee 2}, we get
\begin{equation} v'=(1-\frac{\alpha}{2})u^2, \label{eq: v'}  \end{equation}
so $v$ is a Lyapunov function of \eqref{eq: McGehee 2}, i.e. it is non-decreasing along any orbit.

By \eqref{eq: energy 0}, $\mf$ is a $2$-dim torus homeomorphic to $\mathbb{S}^1 \times \mathbb{S}^1$. We introduce a global coordinates $(\psi, \tht)$ with $\tht$ as above and $\psi$ as 
\begin{equation} \label{eq: uv psi}
\cos\psi = \frac{u}{\sqrt{2\uh(\theta)}},\quad \sin\psi = \frac{v}{\sqrt{2\uh(\theta)}}.  \end{equation} 
Then on $\mf$, the vector field \eqref{eq: McGehee 2} has the following expression:
\begin{equation}
\begin{cases}
\psi'&=(1-\frac{\alpha}{2}) \sqrt{2\uh(\theta)} \cos\psi - \frac{\uh_{\tht}(\theta)\sin\psi}{\sqrt{2\uh(\theta)}},\\
  \theta'&= \sqrt{2\uh(\theta)} \cos\psi. 
\end{cases} \label{eq: vector field collision mfd}
 \end{equation}
\begin{lem} \label{lem: equail} 
\begin{enumerate}
\item[(a).] $(\psi_0, \tht_0) \in \mf$ is an equilibrium of \eqref{eq: vector field collision mfd}, if and only if $\psi_0 \in \{ \pm \pi/2 \}$ and $\tht_0$ is a critical point of $\uh$; 
\item[(b).] If $(\psi, \tht)(\tau)$, $\tau \in \rr$, is a non-equilibrium solution of \eqref{eq: vector field collision mfd}, then $\{\tau \in \rr: \; \tht'(\tau)=0 \}$ is an isolated set in $\rr$. 
\end{enumerate} 
\end{lem}

 
Consider the linearization of \eqref{eq: vector field collision mfd} at an equilibrium $(\psi_0, \tht_0) \in \mf$:
\begin{equation} \label{eq: linear equil}
M(\psi_0, \tht_0)=
 \left( \begin{array}{cc} 
(\frac{\al}{2}-1) \sqrt{2\uh(\tht_0)} \sin \psi_0 & -\frac{\uh_{\tht \tht}(\tht_0) \sin \psi_0}{\sqrt{2 \uh(\tht_0)}} \\
-\sqrt{2\uh(\tht_0)} \sin \psi_0 & 0
\end{array}\right).\end{equation}

\begin{nota} \label{dfn: eigen value vector}
We set $\lmd_\pm(\psi_0, \tht_0)$ as the two eigenvalues of $M(\psi_0, \tht_0)$, and $e_\pm(\psi_0, \tht_0)$  the corresponding eigenvectors. If $\lmd_{\pm}(\psi_0, \tht_0)$ are real numbers, we always assume $\lmd_-(\psi_0, \tht_0) \le \lmd_+(\psi_0, \tht_0)$. When there is no confusion, we may omit $(\psi_0, \tht_0)$ in these notations.
\end{nota}

For $\Delta(\tht_0)$ given in \eqref{eq: delta}, whenever it is negative, $\sqrt{\Delta(\tht_0)}$ should be understood as the imaginary number $i \sqrt{|\Delta(\tht_0)|}$.

\begin{lem} \label{lem: equil eigenvalue} Following the notations given as above, we have 
$$ \lmd_\pm =  -\frac{(2-\al)}{4}\sqrt{2\uh(\tht_0)} \sin \psi_0 \pm \ey \sqrt{\Delta(\tht_0)}; $$
$$ e_\pm = \left(\frac{2-\al}{4} \mp \frac{\sin \psi_0}{2} \sqrt{\frac{\Delta(\tht_0)}{2\uh(\tht_0)}}, 1 \right)^T.$$ 
Furthermore, 
\begin{enumerate}
\item[(a).] when $\uh_{\tht \tht}(\tht_0)>0$, $\Delta(\tht_0)>0$ and $ \lmd_-<0 < \lmd_+;$
\item[(b).] when $\uh_{\tht \tht}(\tht_0)=0$, $\Delta(\tht_0)>0$ and 
$$ \lmd_- < \lmd_+=0, \text{ if } \psi_0=\pi/2; \; 0 = \lmd_- < \lmd_+, \text{ if } \psi_0=-\pi/2;$$
\item[(c).] when $0> \uh_{\tht \tht}(\tht_0) > -\frac{(2-\al)^2}{8}\uh(\tht_0)$, $\Delta(\tht_0)>0$ and 
$$ \lmd_- < \lmd_+ <0, \text{ if } \psi_0=\pi/2; \; 0< \lmd_- < \lmd_+, \text{ if } \psi_0=-\pi/2;$$
\item[(d).] when $\uh_{\tht \tht}(\tht_0)= -\frac{(2-\al)^2}{8}\uh(\tht_0)$, $\Delta(\tht_0)=0$ and
$$ \lmd_- = \lmd_+<0, \text{ if } \psi_0=\pi/2; \; \lmd_- = \lmd_+>0, \text{ if } \psi_0=-\pi/2;$$ 
\item[(e).] when $\uh_{\tht \tht}(\tht_0)< -\frac{(2-\al)^2}{8}\uh(\tht_0) $, $\Delta(\tht_0)<0$ and 
$$ \Re(\lmd_{\pm}) <0, \text{ if } \psi_0=\pi/2; \; \Re(\lmd_{\pm})>0, \text{ if } \psi_0=-\pi/2. $$
\end{enumerate}
\end{lem}

\begin{figure}
  \centering
  \includegraphics[scale=0.48]{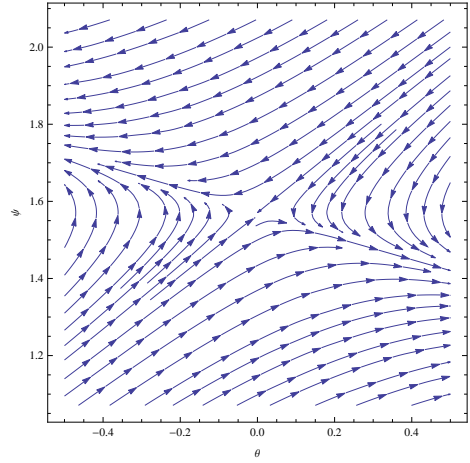}
  \caption{$(\psi_0, \theta_0)=(\pi/2,0)$}
  \label{fig: equil1}
\end{figure}

\begin{figure}
    \centering
    \begin{subfigure}[b]{0.48\textwidth}
        \includegraphics[width=\textwidth]{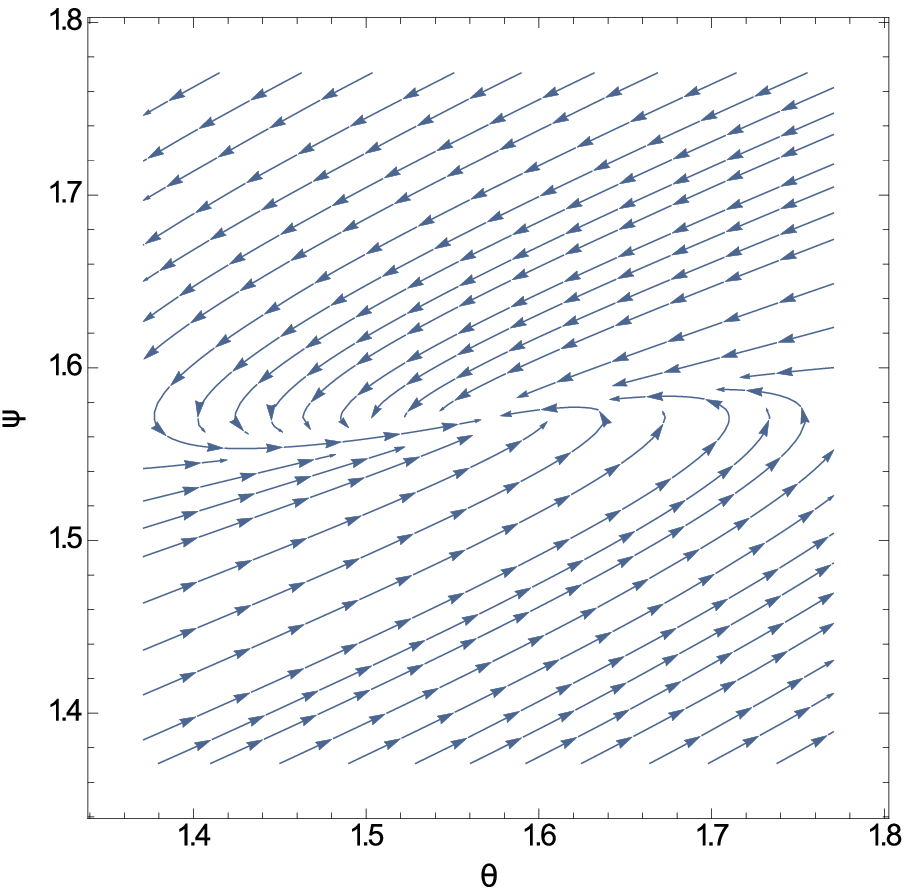}
        \caption{$(\psi_0, \tht_0)=(\pi/2, \pi/2), \Delta>0$}
        \label{fig: equil2}
    \end{subfigure}
    ~
    \begin{subfigure}[b]{0.48\textwidth}
        \includegraphics[width=\textwidth]{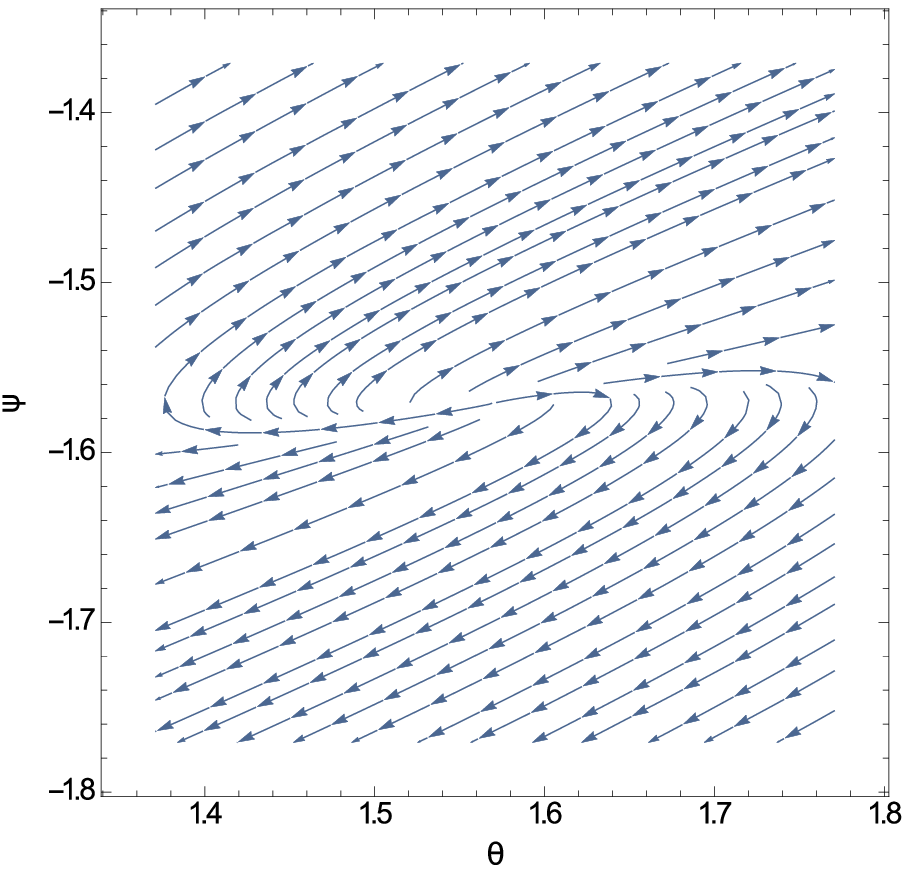}
        \caption{$(\psi_0, \tht_0)=(-\pi/2, \pi/2), \Delta>0$}
        \label{fig: equil3}
    \end{subfigure} \\
    ~    
    \begin{subfigure}[b]{0.48\textwidth}
        \includegraphics[width=\textwidth]{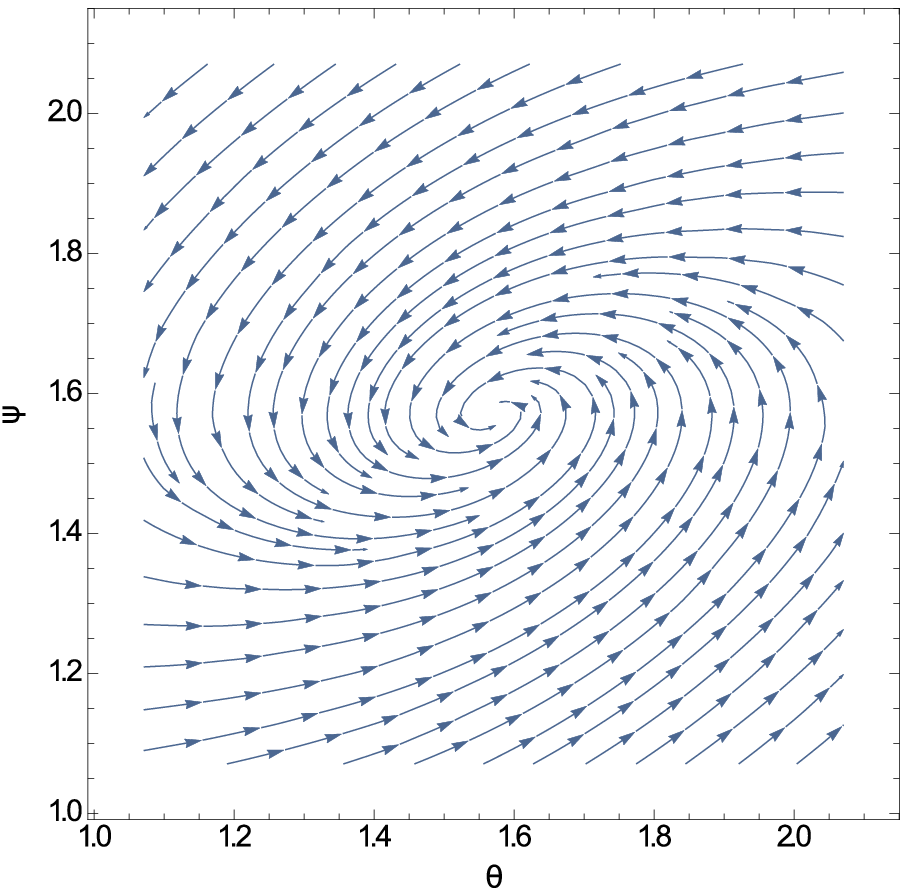}
        \caption{$(\psi_0, \tht_0)=(\pi/2, \pi/2), \Delta<0$}
        \label{fig: equil4}
    \end{subfigure}
    ~
    \begin{subfigure}[b]{0.48\textwidth}
        \includegraphics[width=\textwidth]{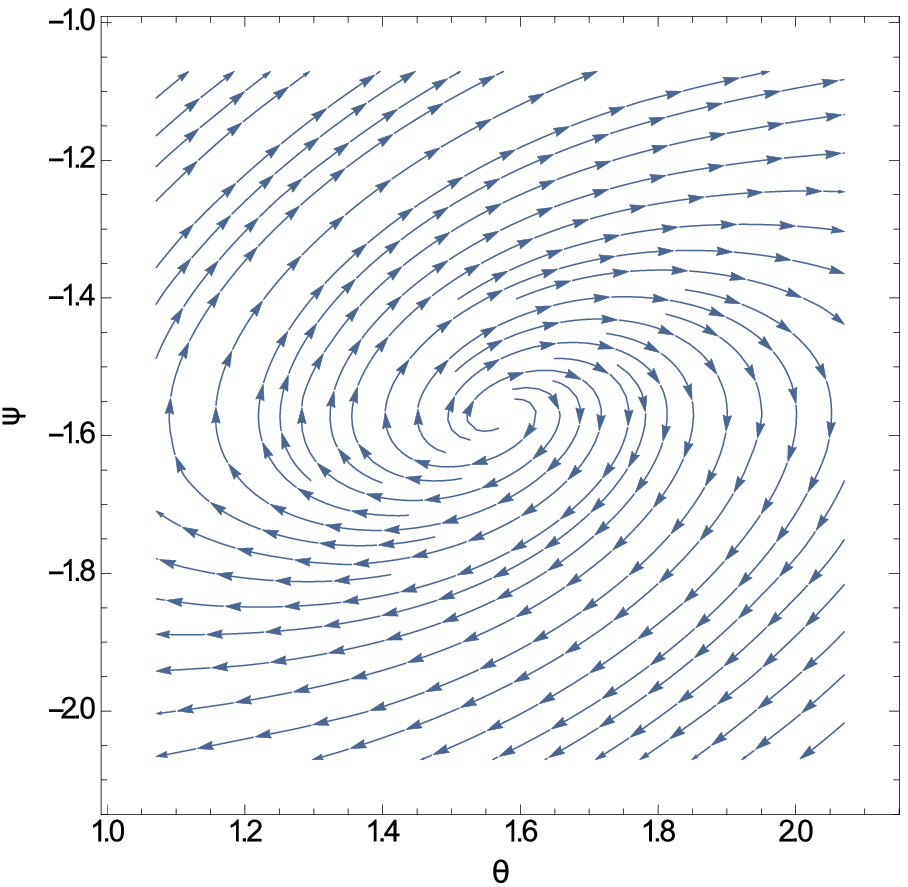}
        \caption{$(\psi_0, \tht_0)=(-\pi/2, \pi/2), \Delta<0$}
        \label{fig: equil5}
    \end{subfigure}
    \caption{}
\end{figure}

The following result is well-known, for a proof see \cite{ZDHD}.
\begin{lem} 
\label{lem: equi locally} When $\tht_0$ is a \emph{non-degenerate} critical point of $\uh$. Then
\begin{enumerate}
 \item[(a).] If $\lmd_-<0<\lmd_+$, then $(\psi_0, \theta_0)$ is a \textbf{saddle}, with a $1$-dim stable manifold and a $1$-dim unstable manifold, which are tangent of linear subspace $\langle e_- \rangle$ and $\langle e_+ \rangle$ at $(\psi_0, \tht_0)$ respectively. See Figure \ref{fig: equil1}. 

\item[(b).] If $\lmd_-<\lmd_+<0$, then $(\psi_0, \theta_0)$ is a \textbf{stable node}. It is asymptotically stable with all the orbits asymptotically converge to $(\psi_0, \tht_0)$, when $t$ goes to positive infinity, along the linear subspace $\langle e_+ \rangle$, except two orbits which asymptotically converge to $(\psi_0, \tht_0)$ along the linear subspace $\langle e_- \rangle$. See Figure \ref{fig: equil2}.

\item[(c).] If $0<\lmd_-<\lmd_+$, then $(\psi_0, \theta_0)$ is a \textbf{unstable node}. It is asymptotically unstable with all the orbits asymptotically converge to $(\psi_0, \tht_0)$, when $t$ goes to negative infinity, along the linear subspace $\langle e_+ \rangle$, except two orbits which asymptotically converge to $(\psi_0, \tht_0)$ along the linear subspace $\langle e_- \rangle$. See Figure \ref{fig: equil3}.

\item[(d).] If $\lmd_{\pm} \in \cc \setminus \rr$, with $\Re(\lmd_{\pm})<0$, then $(\psi_0, \theta_0)$ is a \textbf{stable focus}. It is asymptotically stable with all the orbits spiral into $(\psi_0, \tht_0)$. See Figure \ref{fig: equil4}.

\item[(e).] If $\lmd_{\pm} \in \cc \setminus \rr$, with $\Re(\lmd_{\pm})>0$, then $(\psi_0, \theta_0)$ is a \textbf{unstable focus}. It is asymptotically unstable with all the orbits spiral away from $(\psi_0, \theta_0)$. See Figure \ref{fig: equil5}.
\end{enumerate} 
\end{lem}



\begin{figure} 
    \centering
     \begin{subfigure}[b]{0.48\textwidth}
      \includegraphics[width=\textwidth]{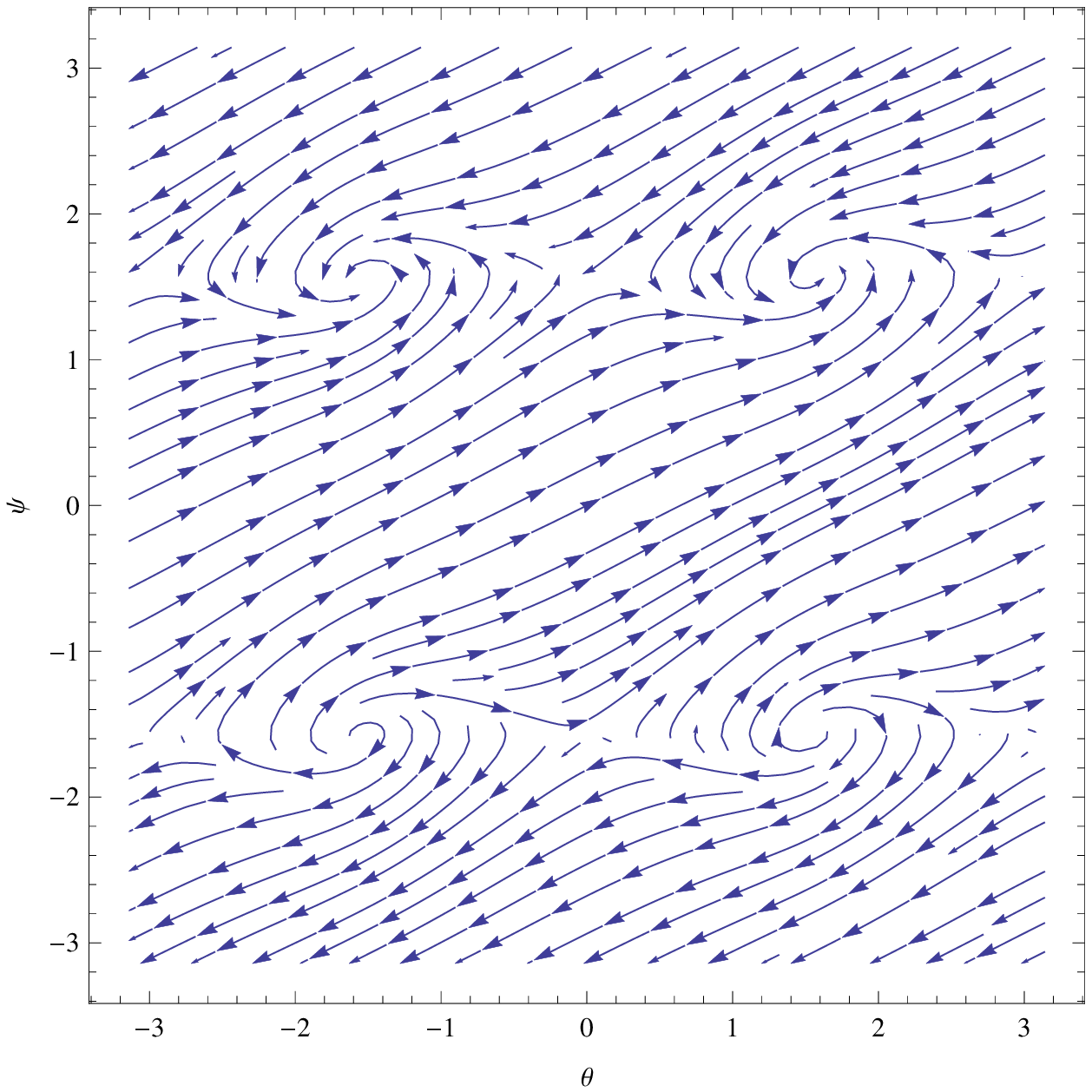}
        \caption{$\al=1, \mu=2$}
    \end{subfigure}
    ~
    \begin{subfigure}[b]{0.48\textwidth}
        \includegraphics[width=\textwidth]{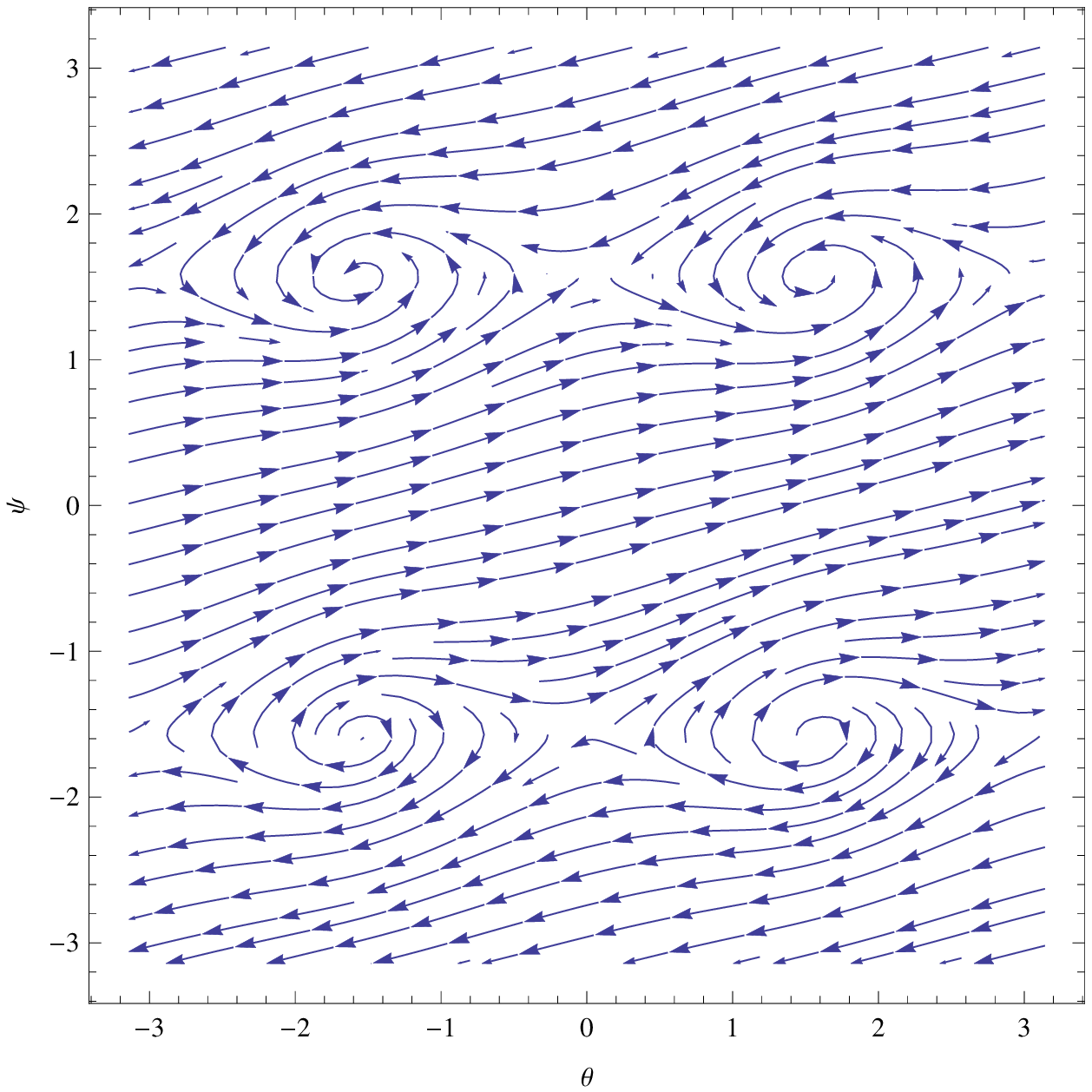}
        \caption{$\al=3/2, \mu=2$}
    \end{subfigure}
    \caption{}
    \label{fig: phase}
\end{figure}

Since $v$ is a Lyapunov function of the vector field on the collision manifold, besides the equilibria, there are no closed or recurrent orbits. As a result

\begin{cor}
\label{cor: dynamics coll mfd} If the critical point of $\uh$ are isolated, any orbit in $\mf$ is either an equilibrium or a heteroclinic orbit connecting two different equilibria.  
\end{cor}

Lemma \ref{lem: equail}, \ref{lem: equil eigenvalue} and \ref{lem: equi locally} give us a complete picture of the phase portraits of the vector field on $\mf$ (see Figure \ref{fig: phase} for numerical pictures when the potential is defined as in \eqref{eq: uh anisotropic kepler}). Let $(\psi, \tht)(\tau)$ be a heteroclinic orbit and $(\psi_0^{\pm}, \tht_0^{\pm})$ two equilibria in $\mf$ satisfying
\begin{equation} \label{eq: limit psi tht}
\lim_{\tau \to \pm \infty}(\psi, \tht)(\tau) = (\psi_0^{\pm}, \tht_0^{\pm}),
\end{equation}
then correspondingly
\begin{equation}
\label{eq: lim v u tht} \lim_{\tau \to \pm \infty}(v, u, \tht)(\tau)= (\sqrt{2\uh(\tht_0^{\pm})} \sin \psi_0^\pm, \sqrt{2\uh(\tht_0^\pm)} \cos \psi_0^\pm, \tht_0^\pm).
\end{equation}
Since $v$ is a Lyapunov function, 
\begin{equation}
\label{eq: v- < v+} \sqrt{2\uh(\tht_0^{-})} \sin \psi_0^-< \sqrt{2\uh(\tht_0^{+})} \sin \psi_0^+. 
\end{equation}
As a result, there are three different types of heteroclinic orbits in $\mf$:
\begin{enumerate}
\item[type-I.] $\psi_0^-= -\pi/2, \psi_0^+= \pi/2$;
\item[type-II.] $\psi_0^-=\psi_0^+=\pi/2$ and $\uh(\tht_0^-)< \uh(\tht_0^+)$;
\item[type-III.] $\psi_0^- = \psi_0^+=-\pi/2$ and $\uh(\tht_0^-) > \uh(\tht_0^+)$. 
\end{enumerate}

\begin{lem}
\label{lem: heteroclinic types} Given a heteroclinic orbit $(\psi, \tht)(\tau)$ in $\mf$,  if $(r, t)(\tau)$ satisfies
\begin{equation} \begin{cases}
r' & = r \sqrt{2 \uh(\tht)} \sin \psi,\\
t' & = r^{1+\frac{\al}{2}},
\end{cases}
\end{equation}
then 
\begin{enumerate}
\item[(a).] when $(\psi, \tht)(\tau)$ is type-I, $\lim_{\tau \to \pm \infty} r(\tau) = \pm \infty$, $\lim_{t \to \pm \infty}t(\tau)= \pm \infty$;
\item[(b).] when $(\psi, \tht)(\tau)$ is type-II, $\lim_{\tau \to +\infty} r(\tau) = +\infty$, $\lim_{\tau \to +\infty}t(\tau) = +\infty$, and $\lim_{\tau \to -\infty} r(\tau)=0$, $\lim_{\tau \to -\infty} t(\tau) = T^-_0 >-\infty$;
\item[(c).] when $(\psi, \tht)(\tau)$ is type-III, $\lim_{\tau \to +\infty} r(\tau) = 0$, $\lim_{\tau \to +\infty}t(\tau) = T^+_0< +\infty$ and $\lim_{\tau \to -\infty} r(\tau)=+\infty$, $\lim_{\tau \to -\infty}t(\tau) = -\infty$.
\end{enumerate}
\end{lem}

Let $x(t)$, $t \in (T^-, T^+)$, be a zero energy solution of \eqref{eq: Lag equation} and $z(\tau)$, $\tau \in \rr$, the corresponding orbit of \eqref{eq: Hamiltonian equation}, we define $\pi(z)(\tau):= (\psi, \tht)(\tau)$ as the projection of $z(\tau)$ in the collision manifold.  
\begin{prop}
\label{prop: parabolic to heteroclinic} If the critical points of $\uh$ are isolated in $\mathbb{S}^1$, then
\begin{enumerate}
\item[(a).] $\pi(z)(\tau)$ is an equilibrium in $\mf$, if and only if $x(t)$ is homothetic; 
\item[(b).] $\pi(z)(\tau)$ is a type-I heteroclinic orbit, if and only if $x(t)$ is a non-homothetic parabolic solution; 
\item[(c).] $\pi(z)(\tau)$ is a type-II heteroclinic orbit, if and only if $x(t)$ is a non-homothetic collision-parabolic solution; 
\item[(d).] $\pi(z)(\tau)$ is a type-III heteroclinic orbit, if and only if $x(t)$ is a non-homothetic parabolic-collision solution. 
\end{enumerate}
\end{prop}
 
\begin{rem}\label{rmk:isolated time} 
The above proposition implies Theorem \ref{thm: zero energy} and for a non-homothetic zero energy solution $x(t)$, Lemma \ref{lem: equail} and Proposition \ref{prop: parabolic to heteroclinic} imply $\{ t \in (T^-, T^+): \; \dot{\tht}(t)=0 \}$ is isolated in $(T^-, T^+)$, as $dt= r^{1+\frac{\al}{2}}d \tau$ and $r(\tau)>0$, for any $\tau$.
\end{rem}

\begin{rem}
For systems with arbitrary finite degrees of freedom, one can define the McGehee coordinates similarly with the corresponding $v$ being a Lyapunov function. Moreover the connection between zero energy solutions and orbits on the collision manifold still exist, so we expect results from this section will still hold. 
\end{rem}

\section{Asymptotic analysis of the linear Hamiltonian system} \label{sec: asymptotic linear}

Throughout this section let $x(t), t \in (T^-, T^+)$, be a non-homothetic zero energy solution of \eqref{eq: Lag equation} and  $z(t)=(p_1, p_2, r, \tht)^T(t)$ the corresponding zero energy orbit of \eqref{eq: Hamiltonian equation}. Consider the linearized equation of \eqref{eq: Hamiltonian equation} along $z(t)$ 
\begin{equation} \dot{\xi}(t)=J\nabla^2H(z(t))\xi(t). \label{eq: linearied} \end{equation}
Under the time parameter $\tau$ (notice that $\tau \to \pm \infty$, as $t \to T^{\pm}$),
\begin{equation} \xi'(\tau)= J B(\tau) \xi(\tau):=r^{1+\frac{\alpha}{2}}(\tau)J\nabla^2 H(z(\tau))\xi(\tau), \label{eq: linearied tau} \end{equation} 
where 
\begin{equation} \label{eq: B tau} B(\tau)=\left(
  \begin{array}{cccc}
    r^{1+\frac{\alpha}{2}} &   0 & 0 &   0\\
    0   & r^{\frac{\alpha}{2}-1} &   -2 r^{\frac{\alpha}{2}-2}p_2& 0\\
    0 &   -\frac{2p_2}{r^{2-\frac{\alpha}{2}}} &  \frac{3p^2_2}{r^{3-\frac{\alpha}{2}}}-\frac{\al(\al+1)\uh(\theta)}{r^{1+\frac{\alpha}{2}}} &   \alpha r^{-\frac{\alpha}{2}}\uh_{\tht}(\theta)\\
    0   &0 &    \alpha r^{-\frac{\alpha}{2}}\uh_{\tht}(\theta) & -r^{1-\frac{\alpha}{2}}\uh_{\tht \tht}(\theta)  \\
  \end{array}
\right) (\tau).
\end{equation}

Our main goal is to understand the asymptotic behavior of the above linear Hamiltonian system, as $\tau$ goes to $\pm \infty$. To separate the variable $r$, we define the following symplectic matrix 
\begin{equation} \label{eq: R}
R(\tau) =\text{diag} (r^{\frac{1}{2}+\frac{\alpha}{4}},r^{\frac{\alpha}{4}-\frac{1}{2}},r^{-\frac{1}{2}-\frac{\alpha}{4}},r^{\frac{1}{2}-\frac{\alpha}{4}})(\tau).
\end{equation}
Then for $\xi(\tau)$ satisfying \eqref{eq: linearied tau}, $\eta(\tau)=R(\tau)\xi(\tau)$ is a solution of 
\begin{equation} \eta'(\tau) =J\hat{B}(\tau)\eta(\tau), \label{eq: linearied normalized} \end{equation}
where $\hat{B}(\tau)=-JR'(\tau)R^{-1}(\tau)+R^{-1}(\tau)B(\tau)R^{-1}(\tau)$. 

Under McGehee coordinates \eqref{eq: v u McGehee},
\begin{equation} \hat{B}(\tau)=\left(
  \begin{array}{cccc}
    1 &   0 & -\frac{2+\al}{4}v &   0\\
    0   & 1 &   -2u & \frac{2-\alpha}{4}v\\
    -\frac{2+\alpha}{4}v &   -2u &  3u^2-\alpha(\alpha+1)\uh(\theta) &   \alpha \uh_{\tht}(\theta)\\
    0   & \frac{2-\alpha}{4}v &    \alpha \uh_{\tht}(\theta) & -\uh_{\tht \tht}(\theta)  \\
  \end{array}
\right)(\tau).
 \end{equation} 

Recall that the projection of $z(\tau)$ on the collision manifold, $(\psi, \tht)(\tau)=\pi(z)(\tau)$ is a heteroclinic orbit between two equilibria. Let $T^*=\pm \infty$, then 
\begin{equation}
\label{eq: lim psi tht *} 
\lim_{\tau \to T^*}(\psi, \tht)(\tau)= (\psi^*_0, \tht^*_0), \text{ where } \psi_0^{*} \in \{\pm \pi/2\}, \text{ and }  \uh_{\tht}(\tht_0^*)=0. 
\end{equation}
By \eqref{eq: lim v u tht},
\begin{equation}
\label{eq: lim v u tht type I}  \lim_{\tau \to T^*} (v, u, \tht)(\tau)= (\sin\psi_0^* \sqrt{2\uh(\tht_0^*)}, 0, \tht_0^*).
\end{equation} 

This implies $\bh_{*} := \lim_{\tau \to T^*} \bh(\tau)$ exists. Moreover $\bh_* = \bhy_{*} \diamond \bhe_{*}$ with
$$\hat{B}^{(1)}_* :=\left( 
\begin{array}{cccc}
1 &  -\frac{2+ \alpha}{4} \sqrt{ 2\uh(\theta_0^*)} \sin \psi^*_0 \\
-\frac{2+ \alpha}{4} \sqrt{ 2\uh(\theta^*_0)}\sin \psi^*_0 & -\alpha(\alpha+1)\uh(\theta^*_0)  
\end{array}
\right), $$ 
$$ \hat{B}^{(2)}_* :=\left(
 \begin{array}{cccc}
 1 & \frac{2-\alpha}{4}\sqrt{2\uh(\theta^*_0)}\sin \psi^*_0 \\
\frac{2-\alpha}{4}\sqrt{2\uh(\theta^*_0)}\sin \psi^*_0 & -\uh_{\tht \tht}(\theta^*_0)  
\end{array}
\right).$$ 
The \emph{symplectic sum} $\diamond$ is defined as in \cite{Lon4}: for any two $2m_k\times 2m_k$ square block matrices,  $M_k=\left(\begin{array}{cc}A_k&B_k\\ 
                             C_k&D_k\end{array}\right)$, $k=1, 2$, $ M_1 \diamond M_2=\left(
  \begin{array}{cccc}
   A_1 &   0 & B_1 &   0\\
                            0   & A_2 &   0 & B_2\\
                           C_1 &   0 & D_1 &   0\\
                           0   & C_2 &   0 & D_2  \\
  \end{array}
\right).\nonumber $

For $i =1$ or $2$, let $\lh^i_{\pm}(\psi_0^*,\theta_0^*)$ be the two eigenvalues of $J\bh^{(i)}_{*}$ with $\lh^i_{-}(\psi_0^*,\theta_0^*) \le \lh^i_{+}(\psi_0^*,\theta_0^*)$, when both of them are real, and $\eh^i_{\pm}(\psi^*_0,\theta^*_0)$ the corresponding eigenvectors.  When there is no confusion, we may omit $(\psi^*_0, \tht^*_0)$ in these notations. 
Direct computations give us the following lemma.

\begin{lem}
\label{lem: B hat 1} $J\bhy_{*}$ is a hyperbolic matrix with
$$ \lh^1_{\pm} = \pm\frac{2 +3\al}{4} \sqrt{2 \uh(\tht_0^*)} , \quad \eh^1_{\pm}= \left( \frac{(2+\al)\sin \psi^*_0\pm(2+3\al)}{4} \sqrt{2\uh(\tht^*_0)}, 1 \right)^T.$$ 
 $J\bhe_{*}$ is a hyperbolic matrix, when $\Delta(\tht^*_0)>0$, with 
$$ \lh^2_{\pm}= \pm\ey \sqrt{\Delta(\tht^*_0)},\quad   \eh^2_{\pm}= \left(-\frac{2-\al}{4}\sqrt{2\uh(\tht^*_0)} \sin \psi^*_0\pm\ey \sqrt{\Delta(\tht^*_0)}, 1 \right)^T. $$ \end{lem} 

Since $U$ is $(-\alpha)$-homogeneous, when $x(t)=(r(t),\theta(t))$ is solution of 
\eqref{eq: Lag equation},  so is $x_h(t)=(r_h(t), \theta_h(t)):=(h^{-\frac{2}{2+\alpha}}r(ht), \theta(ht))$,  for any $h>0$. This means 
\begin{equation} \dot{z}_h(t)=J\nabla H_h(z_h(t)), \label{eq: Hamiltonian h}  \end{equation}
where $z_h(t)= (p_{1,h}, p_{2,h}, r_h, \tht_h)^T(t)$ with 
$$ p_{1,h}(t)=\dot{r}_h(t)=h^{\frac{\alpha}{2+\alpha}}\dot{r}(ht); \;\;  p_{2,h}(t)=r^2_h(t)\dot{\theta}_h(t)=h^{\frac{\alpha-2}{\alpha+2}}r^2(ht)\dot{\theta}(ht), $$
and 
$$H_h(z_h(t))= \ey(p_{1,h}^2(t) + r_h^{-2}(t)p^2_{2,h}(t))- r^{-\al}_h(t)\uh(\tht_h(t)).$$
 
Let $h=1$ and differentiate \eqref{eq: Hamiltonian h} with respect to $t$, we get a solution of \eqref{eq: linearied}: 
$$ \zt_1(t):=\dot{z}_1(t)=(\ddot{r}, 2r\dot{r}\dot{\theta}+r^2\ddot{\theta}, \dot{r},\dot{\theta})^T(t).$$
Meanwhile by differentiating \eqref{eq: Hamiltonian h} with respect to $h$, we get
\begin{equation} \frac{d \dot{z}_h}{dh}|_{h=1}(t)=J\nabla^2 H(z_1(t)) \left(\frac{d z_h}{dh}|_{h=1}(t) \right). \label{3.7}  \end{equation}
Hence $\zt_3(t):= \frac{d z_h}{dh}|_{h=1}(t)$ is another solution of \eqref{eq: linearied}. Define
$$\zt_2(t):=\zt_3(t)-t\zt_1(t) =\left(\frac{\alpha}{2+\alpha}\dot{r}, \frac{\alpha-2}{\alpha+2}r^2\dot{\theta},-\frac{\alpha}{2+\alpha}r,0 \right)^T(t).$$


Under the time parameter $\tau$, using $R(\tau)$ given in \eqref{eq: R}, we find the following two solutions of the linear system \eqref{eq: linearied normalized}:
\begin{equation}
\label{eq: eta 1} \eta_1(\tau)= R(\tau)\zt_1(\tau)=r^{-\frac{2+3\al}{4}}(\tau) (u^2-\alpha \uh(\theta), \uh_{\tht}(\theta),v,u)^T(\tau),
\end{equation}
\begin{equation}
\label{eq: eta 2} \eta_2(\tau)= R(\tau)\zt_2(\tau)= r^{\frac{2-\alpha}{4}}(\tau) \left(\frac{\alpha v}{2+\alpha}, \frac{\alpha-2}{\alpha+2}u, -\frac{2}{2+\alpha},0 \right)^T(\tau).
\end{equation}

\begin{defi} \label{defi: V Lag}
For each $\tau \in \rr$, we define $V(\tau):=\text{span}\{\eta_1(\tau),\eta_2(\tau)\}$
 as the linear space generated by $\eta_1(\tau)$ and $\eta_2(\tau)$ defined as above.
\end{defi}
Notice that $\eta_1(\tau)$ and $\eta_2(\tau)$ are linear independent if and only if $x(t)$ is a non-homothetic solution.

Let $(\mathbb{R}^4,\omega)$ with $\omega(x,y)=(Jx,y) $ being the standard symplectic form on $\rr^4$. A subspace $V \subset \rr^2$ is \emph{Lagrangian}, if $\text{dim}(V)=2$ and $\omega|_V=0$.  We denote by $\text{Lag}(\rr^4)$ the Lagrangian Grassmannian, i.e. the set of all Lagrangian subspaces of $(\rr^4,\omega)$. For any $V \in \text{Lag}(\rr^4)$, let $P_V$ be the orthogonal projection of $\rr^4$ to $V$, then 
$$ \text{dist}(W, W^*):= \|P_W-P_{W^*}\|, \text{ for any } W, W^* \in \text{Lag}(\rr^4), $$
gives a complete metric on $Lag(\rr^4)$. Here $\|\cdot \|$ represents the metric on the space of bounded linear operators from $\rr^4$ to itself.

\begin{lem}\label{lem: V Lag}  
If $x(t)$ is a non-homothetic zero energy solution, $V(\tau)\in C^0(\rr, \textnormal{Lag}(\rr^4))$.\end{lem} 
\begin{proof} By a direct computation,
$$ \om(\eta_1, \eta_2) = \frac{2\al}{2+\al} r^{-\al} \left( \ey(u^2 +v^2) -\uh(\tht)\right).$$
Then the result follows from \eqref{eq: energy} and $x(t)$ with $0$ energy.
\end{proof}

We will study the limit of $V(\tau)$, as $\tau$ goes to $T^*$. For it to exist,  $J\bh_{*}$ needs to be hyperbolic, and the precise limit depends on how the corresponding heteroclinic orbit $(\psi, \tht)(\tau)$ approaches to the equilibrium $(\psi^*_0, \tht^*_0)$ on the collision manifold. 

When $\Delta(\tht_0^*)>0$, by Lemma \ref{lem: equil eigenvalue} and \ref{lem: equi locally}, $(\psi, \tht)(t)$ converges to $(\psi^*_0, \tht^*_0)$ either along the subspace $\langle e^*_- \rangle$ or $\langle e^*_+ \rangle$, where $e^*_{\pm} = e_{\pm}(\psi_0^*, \tht_0^*)$, see Notation \ref{dfn: eigen value vector}. 

\begin{prop}\label{prop: lim V} 
Assume $\Delta(\tht_0^*)>0$ and $\uh_{\tht \tht}(\tht_0^*) \ne 0$, 
when $(\psi, \tht)(\tau) \to (\psi^*_0, \tht^*_0)$ along $\langle e^*_{\pm}\rangle$, as $\tau \to T^*$,  
\begin{equation} \lim_{\tau \to T^*} V(\tau) = \text{span}\{ \eh^1_{j(\psi^*_0)}, \eh^2_{\pm}\}, \; \text{ where } \; \eh^1_{j(\psi^*_0)}=\begin{cases}
\eh^1_+ & \text{ if} \; \psi^*_0=-\pi/2\\
\eh^1_- & \text{ if} \; \psi^*_0=\pi/2
\end{cases}.
\end{equation}

\end{prop}

We first give a proof of the above proposition using the following lemma. 

\begin{lem} \label{lem: lim uh_tht u} 
 Assume $\Delta(\tht^*_0) >0$ and $\uh_{\tht \tht}(\tht^*_0) \ne 0$,  if $(\psi, \tht)(\tau) \to (\psi^*_0, \tht^*_0)$ along $\langle e_{\pm}^* \rangle$, as $\tau \to T^*$,  then 
 $$\lim_{\tau \to T^*} \frac{\uh_{\tht}(\tht)}{u} =-\lmd_\mp(\psi^*_0, \tht^*_0) =  \frac{2-\al}{4}\sqrt{2\uh(\tht^*_0)} \sin \psi^*_0 \pm \ey \sqrt{\Delta(\tht^*_0)}.$$

 \end{lem}

\begin{proof}
 We only give details for $\psi_0^*=\pi/2$ and $(\psi, \tht)(\tau)$ converges to $(\psi_0^*, \tht^*_0)$ along $\langle e^*_- \rangle$, while the others are similarly. Let $\e_i \in \rr^4$, $i=1,2,3,4$, be an orthogonal basis of $\rr^4$ with the $i$-th component equal to $1$ and the others all being zero.,

 Let $V(\tau) =\text{span}\{\eta_1(\tau), \eta_2(\tau)\}$ be defined as in Definition \ref{defi: V Lag}, then
 \begin{align*}
 \eta_1 \wg \eta_2 & = \frac{u}{(2+\al)r^{\al}} \Big\{ \Big( (2-\al)(\al \uh(\tht)-u^2)-\al v \frac{\uh_{\tht}(\tht)}{u}\Big) \e_1 \wg \e_2  \\
 & -(2-\al)u \e_1 \wg \e_3 -\al v \e_1 \wg \e_4 + \Big( (2-\al)v-2 \frac{\uh_{\tht}(\tht)}{u}\Big) \e_2 \wg \e_3 \\ 
 &+ (2-\al)u \e_2 \wg \e_4 + 2 \e_3 \wg \e_4
 \Big\}.
 \end{align*}
 By \eqref{eq: lim v u tht type I} and  Lemma \ref{lem: lim uh_tht u}, a direct computation shows

 \begin{multline}
 \label{eq: lim eta positive inf}
 \lim_{\tau \to T^*} \frac{(2+\al)r^\al}{u} \eta_1 \wg \eta_2 = \Big( \frac{\al(2-\al)}{2}\uh(\tht_0^*)+ \frac{\al}{2}\sqrt{2\uh(\tht_0^*) \Delta(\tht_0^*)} \Big) \e_1 \wg \e_2  \\
 -\al \sqrt{2\uh(\tht_0^*)} \e_1 \wg \e_4 + \Big( \frac{2-\al}{2}\sqrt{2\uh(\tht_0^*)} + \sqrt{\Delta(\tht_0^*)} \Big) \e_2 \wg \e_3 + 2 \e_3 \wg \e_4. 
 \end{multline}

Meanwhile for $\eh^1_{-}, \eh^2_{-}$ given in Lemma \ref{lem: B hat 1}, a straight forward computation shows $ 2 \eh^1_{-} \wg \eh^2_{-}$ is the same as what we got in \eqref{eq: lim eta positive inf}. This finishes our proof.
\end{proof}

 \begin{proof}[Proof of Lemma \ref{lem: lim uh_tht u}]
 We will only give the details for $T^*= +\infty$. 
 As both $\uh_{\tht}(\tht)$ and $u$ goes to $0$, when $\tau \to +\infty$, by L'Hospital's rule,
\begin{equation}
\label{eq: posi inf} \lim_{\tau \to +\infty} \frac{\uh_{\tht}(\tht)}{u} = \lim_{\tau \to +\infty} \frac{\uh_{\tht \tht}(\tht)\tht'}{(\sqrt{2\uh(\tht)} \cos \psi)'}= -\frac{\uh_{\tht \tht}(\tht^*_0)}{\sqrt{2 \uh(\tht^*_0)} \sin \psi^*_0} \lim_{\tau \to +\infty} \frac{\tht'}{\psi'} 
\end{equation}
If $(\psi, \tht)(\tau)$ converges to $(\psi^*_0, \tht^*_0)$ along $\langle e^*_- \rangle$, as $\tau \to +\infty$, by Lemma \ref{lem: equil eigenvalue},
$$ \lim_{\tau \to +\infty} \frac{\psi'}{\tht'} = \frac{2-\al}{4}+ \frac{\sin \psi^*_0}{2} \sqrt{\frac{\Delta(\tht^*_0)}{2 \uh(\tht^*_0)}}. $$
Plug this into \eqref{eq: posi inf}, we get 
$$ \lim_{\tau \to +\infty} \frac{\uh_{\tht}(\tht)}{u} = \frac{2-\al}{4}\sqrt{2\uh(\tht^*_0)} \sin \psi^*_0 - \ey \sqrt{\Delta(\tht^*_0)} = -\lmd_+(\psi_0^*, \tht_0^*). $$
The second equality follows from Lemma \ref{lem: equil eigenvalue}. 

Similarly if $(\psi, \tht)(\tau)$ converges to $(\psi_0^*, \tht_0^*)$ along $\langle e^*_+ \rangle$, as $\tau \to +\infty$, then 
$$ \lim_{\tau \to +\infty} \frac{\psi'}{\tht'} = \frac{2-\al}{4}- \frac{\sin \psi_0^*}{2} \sqrt{\frac{\Delta(\tht_0^*)}{2 \uh(\tht_0^*)}},$$
and
$$ \lim_{\tau \to +\infty} \frac{\uh_{\tht}(\tht)}{u} = \frac{2-\al}{4}\sqrt{2\uh(\tht^*_0)} \sin \psi^*_0 + \ey \sqrt{\Delta(\tht^*_0)}= -\lmd_-(\psi^*_0, \tht^*_0). $$
\end{proof}

\section{Connect the Morse and oscillation indices by Maslov indices} \label{sec: Morse Maslov indices}
In this section, except the last proof, which deals with the homothetic solution, we always assume $x(t)$, $t \in (T^-, T^+)$,  is a non-homothetic zero energy solution of \eqref{eq: Lag equation} with $z(t)$ being the corresponding zero energy orbit of \eqref{eq: Hamiltonian equation} and $\pi(z)(\tau)$ the heteroclinic orbit on $\mf$ satisfying $\lim_{\tau \to \pm\infty}\pi(z)(\tau)= (\psi_0^{\pm}, \tht_0^{\pm}).$

We need the \emph{Maslov index} to connect the Morse and oscillation indices. For details of the Maslov index, see \cite{CLM} or Section \ref{sec: appendix B}. Let $\gamma(t, t_1)$ be the fundamental solution of the linear Hamiltonian equation \eqref{eq: linearied}:
\begin{equation}  
\dot{\gamma}(t,t_1)=J\nabla^2 H(z(t))\gamma(t,t_1),\quad  \gamma(t_1,t_1)=I_4. 
\end{equation}
 For any $t_1<t_2$, we define the Maslov index of $x(t), t \in [t_1, t_2]$ as 
\begin{equation}
\label{eq: Maslov index} \mu(V_d, \gm(t, t_1)V_d; [t_1, t_2]), \; \text{ where } V_d:=\rr^2 \oplus 0.
\end{equation}
By the \emph{Morse Index Theorem} (see \cite{HS}) 
\begin{equation} 
m^-(x; t_1, t_2)+2=\mu(V_d, \gamma(t,t_1)V_d; [t_1,t_2]).
\end{equation}

Under the time parameter $\tau$, the corresponding $\gm(\tau, \tau_1):=\gm(t(\tau), \tau_1)$, where $t_1=t(\tau_1)$, is the fundamental solution of \eqref{eq: linearied tau}, and $\hat{\gamma}(\tau,\tau_1)=R(\tau)\gamma(\tau,\tau_1)R^{-1}(\tau_1)$ ($R(\tau)$ is the matrix defined in \eqref{eq: R}) is the fundamental solution of equation \eqref{eq: linearied normalized}:
\begin{equation} \label{eq: linearized gm hat}
\hat{\gamma}'(\tau,\tau_1)=J\hat{B}(\tau)\hat{\gamma}(\tau,\tau_1), \quad  \hat{\gamma}(\tau_1,\tau_1)=I_4. 
\end{equation}
\begin{lem} When $t_i= t(\tau_i)$, $i=1,2$, 
$$ \mu(V_d, \hat{\gamma}(\tau,\tau_1)V_d; [\tau_1,\tau_2]) =\mu(V_d, \gm(t, t_1)V_d; [t_1, t_2]). $$
\end{lem}
\begin{proof}
First as the Maslov index is invariant under the change of time parameter, 
\begin{equation} 
\mu(V_d, \gm(t, t_1)V_d; [t_1, t_2])=\mu(V_d, \gamma(\tau,\tau_1)V_d; [\tau_1,\tau_2]),  
\end{equation}
Meanwhile
\begin{equation} 
\begin{split}
\mu(V_d, &\hat{\gamma}(\tau,\tau_1)V_d; [\tau_1,\tau_2]) =\mu(V_d, R(\tau)\gamma(\tau,\tau_1)R^{-1}(\tau_1)V_d; [\tau_1,\tau_2])\\ 
 & = \mu(R^{-1}(\tau)V_d, \gamma(\tau,\tau_1)R^{-1}(\tau_1)V_d; [\tau_1,\tau_2]) = \mu(V_d, \gamma(\tau,\tau_1)V_d; [\tau_1,\tau_2]).
\end{split}
\end{equation}  
The last equality follows from the fact that $R^{-1}(\tau)V_d = V_d$, for any $\tau$, as $R(\tau)$ is a diagonal matrix.  
\end{proof}
By the above lemma, 
  \begin{equation} m^-(x; t_1,t_2)+2=\mu(V_d, \hat{\gamma}(\tau,\tau_1)V_d; [\tau_1,\tau_2]). \label{4.8} \end{equation}
Then for any sequences $\tau^-_n < \tau^+_n$ satisfying $\lim_{n \to +\infty}\tau^{\pm}_n = \pm \infty$, 
\begin{equation}
m^-(x)+2= \lim_{n \to +\infty} \mu(V_d, \hat{\gamma}(\tau,\tau_1)V_d; [\tau^-_n,\tau^+_n]).
\end{equation}

To compute the above limit, we need another Maslov index. For any $\tau \in \rr$, define the stable/unstable subspace   $V^+(\tau)/V^-(\tau)$ of the linear system \eqref{eq: linearized gm hat} as 
\begin{align*}
 V^\pm(\tau) & := \{ v \in \rr^4| \; \lim_{\sg \to \pm\infty} \gh(\sg, \tau) v= 0  \}.
\end{align*}
Notice that $V^{\pm}(\tau) = \hat{\gm}(\tau, \sg) V^{\pm}(\sg)$, for any two $\sg, \tau  \in \rr.$ 

\begin{defi} \label{dfn: Maslov index} We define the \textbf{Maslov index} $\mu(x)$ of $x$ as
\begin{equation} 
\mu(x):=\mu(V_d,V^-(\tau); \mathbb{R}) = \lim_{T \to +\infty} \mu(V_d, V^-(\tau); [-T, T]).  
\end{equation}  
\end{defi} 
The index $\mu(x)$ defined above was introduced in the study of heteroclinic orbits (see \cite{HO}, \cite{HP} or the Appendix for more details).

At this moment it is not clear whether $\mu(x)$ is well defined. We will show this shortly. Following the notations from the previous section, we set $T^*= \pm \infty$. Recall that $J\bh_{*}= \lim_{\tau \to T^*} J\bh(\tau)$ is a hyperbolic matrix, when $\Delta(\tht^*_0) >0$. Let $V^{+}(J\bh_{*})$ and $V^-(J\bh_{*})$ be the $J\bh_{*}$ invariant  subspaces of $\rr^4$ corresponding to eigenvalues with positive and negative real part respectively. By Lemma \ref{lem: B hat 1}, $$ V^\pm(J\bh_{*}) = \text{span}\{\eh^1_\pm, \eh^2_\pm\}. $$

In the following, we may need to specify the value of $T^*$, in those cases we set $J\bh_{\pm}:= \lim_{\tau \to \pm \infty} J\bh(\tau)$. The next lemma follows from \cite[Theorem 2.1]{AM}.

\begin{lem}
\label{lem: Vs Vu} When $J\bh_{\pm}$ are hyperbolic matrices.
\begin{enumerate}
\item[(a).] $V^\pm(\tau)$ is the only linear subspace of $\rr^4$ satisfying $\gh(\sg, \tau) V^\pm(\tau) \to V^\mp(J\bh_\pm)$, as $\sg \to \pm\infty$. 
\item[(b).] if a linear subspace $W$ of $\rr^4$ is topologically complement of $V^+(\tau)$ ( or $V^-(\tau)$), then for any $v \in W \setminus \{0 \}$, $|\gh(\sg, \tau) v| \to +\infty$ exponentially fast, as $\sg \to +\infty$ ( or $\tau \to -\infty$), and $\gh(\sg, \tau)W \to V^+(J\bh_+)$ ( or  $V^-(J\bh_-)$) at the same time. 
\end{enumerate}
\end{lem} 

In the following, for $i =1,2$,  let $\eta_i(\tau),$ $\tau \in \rr$, be two solutions of the linear system \eqref{eq: linearied normalized} given in \eqref{eq: eta 1} and \eqref{eq: eta 2}, and $V(\tau) = \text{span}\{\eta_1(\tau), \eta_2(\tau)\}$ is defined in Definition \ref{defi: V Lag}. Since $x(t)$ is non-homothetic,    
 By Lemma \ref{lem: V Lag}, $V(\tau)\in C^0(\rr, \text{Lag}(\rr^4))$.
\begin{lem} \label{lem: W tau}  
Assume $J\bh_{\pm}$ are hyperbolic and $W(\tau)\in C^0(\rr, \text{Lag}(\rr^4))$ is invariant under the flow of \eqref{eq: linearized gm hat}, if $\eta_1(\tau) \in W(\tau)$, $\forall \tau \in \rr$, then
  $$W(T^*):=\lim_{\tau \to T^*} W(\tau) = \text{span}\{\hat{e}^1_{j(\psi_0)}, \hat{e}^2_{-}\} \; \text{or} \; \text{span}\{\hat{e}^1_{j(\psi_0)}, \hat{e}^2_{+}\},$$
where $\eh^1_{j(\psi^*_0)}=\eh^1_+$, if $ \psi^*_0=-\pi/2$ and $\eh^1_{j(\psi^*_0)}=
\eh^1_- $, if $\psi^*_0=\pi/2. $ 
\end{lem}

\begin{proof} We will only give the details for $T^* =+\infty$ and $\psi^+_0=\pi/2$. The others are similar. 
Recall that $\eta_1 = r^{-\frac{2+3\al}{2}}(u^2-\al \uh(\tht), \uh_{\tht}(\tht), v, u)^T$. Then 
\begin{equation}
\label{eq: lim eta 1} \lim_{\tau \to +\infty} \frac{\eta_1(\tau)}{|\eta_1(\tau)|} = \frac{(-\frac{\al}{2}\sqrt{2\uh(\tht_0^+)}, 0, 1, 0)^T}{\sqrt{\al^2 \uh(\tht_0^+)/2 +1}} = \frac{\eh^1_{-}}{\sqrt{\al^2 \uh(\tht_0^+)/2 +1}}.
\end{equation}
As $\psi^+_0=\pi/2$, $\pi(z)(\tau)$ is either a type-I or type-II heteroclinic orbit. By Lemma \ref{lem: heteroclinic types}, $\lim_{\tau \to +\infty} r(\tau)=+\infty$, which implies $\lim_{\tau \to +\infty}\eta_1(\tau)=0$. 

Since $W(\tau)$ is a Lagrangian subspace, we can always find another path $\eta(\tau) \in W(\tau), \tau \in \rr$, invariant under the flow of \eqref{eq: linearized gm hat}, independent of $\eta_1(\tau)$ and satisfying $\eta(\tau) \in V^\omega(\eta_1(\tau))$, i.e. the $\omega$ orthogonal space of $\eta_1(\tau)$ in $\rr^4$.

If $ \lim_{\tau \to +\infty} \eta(\tau)=0$, since the dimension of $W(\tau)$ is two, $\lim_{\tau \to +\infty} \gm(\sg, \tau)v =0$, for any $v \in W(\tau)$. By Lemma \ref{lem: Vs Vu}, $W(\tau) = V^+(\tau)$ and $W(\tau) \to V^-(J\bh_+)= \text{span}\{\eh^1_{-}, \eh^2_{-}\},$ when $\tau \to +\infty$.  

If $ \lim_{\tau \to +\infty} \eta(\tau) \ne 0$, then we can find a two dimensional linear subspace of $\rr^4$, which is a topological complement of $V^+(\tau)$ and contains $\eta(\tau)$. By Lemma \ref{lem: Vs Vu}, for $\tau$ large enough, $\eta(\tau)\in \mathcal{N}_\vep(V^+(J\bh_+))$, i.e. the $\vep$ neighborhood of  $V^+(J\bh_+)$, for some $\vep>0$ small enough. As a result, $\eta(\tau) \in \mathcal{N}_\vep(V^+(J\bh_+))\cap V^\omega(\eta_1)$. Since $\vep$ can be arbitrarily small, $\lim_{\tau \to +\infty} \eta(\tau)/|\eta(\tau)| \in V^+(J\bh_+)\cap V^\omega(\eh^1_{-})$. 

Recall that $V^+(J\bh_+)= \text{span}\{\eh^1_{+}, \eh^2_{+}\}$, by a direct computation,
$$ \om(\eh^1_{-}, \eh^2_{+}) =0, \;\; \om(\eh^1_{-}, \eh^1_{+}) = -\frac{2+3\al}{2}\sqrt{2\uh(\tht_0^+)} \ne 0. $$ 
Hence $\lim_{\tau \to +\infty} \eta(\tau)/|\eta(\tau)|=\eh^2_{+}/ |\eh^2_{+}|$, and together with \eqref{eq: lim eta 1}, it shows $W(\tau) \to \text{span}\{\hat{e}^1_{-}, \hat{e}^2_{+}\}$ as $\tau\to +\infty. $ This finishes our proof. \end{proof} 



Under the assumption of Lemma \ref{lem: W tau}, $W(T^*)\pitchfork V_d$ ($\pitchfork$ means transversal intersection).  Notice that $ \eta_1(\tau) \in V^+(\tau) $ (or $V^-(\tau)$) implies $V^+(T^*)\pitchfork V_d$ (or $V^-(T^*)\pitchfork V_d$), where $V^{\pm}(T^*):= \lim_{\tau \to T^*} V^{\pm}(\tau)$. Then Lemma \ref{lem: heteroclinic types} and \ref{lem: W tau} tell us.
\begin{cor}\lb{cor: transversal} \begin{enumerate}  
\item[(a).] If $\pi(z)(\tau)$ is type-I, then $V^\pm(T^*)\pitchfork V_d$.
 \item[(b).] If $\pi(z)(\tau)$ is type-II, then $V^+(T^*)\pitchfork V_d$.
\item[(c).] If $\pi(z)(\tau)$ is type-III, then $V^-(T^*)\pitchfork V_d$. 
\end{enumerate}
\end{cor}

By the above corollary, when $\pi(z)(\tau)$ is type-I or III, $\mu(V_d, V^-(\tau); [-T, T])$ is a constant for $T>0$ large enough, so $\mu(x)$ given in Definition \ref{dfn: Maslov index} is well defined. 

\begin{thm} \label{thm4.1} 
If $\pi(z)(\tau)$ is type-I or III and $\Delta(\tht_0^{\pm})>0$, then $ m^-(x)=\mu(x).$   
\end{thm}
\begin{proof}
When $\vep>0$ is small enough, 
\[\mu(V_d, e^{s \vep J}V_d; s\in[0,1])=2.\]  
Then for any $T>0$ large enough, by \eqref{1.3c.1},
\[\mu(V_d,\hat{\gamma}(T,-T)e^{s\vep J}V_d; s\in[0,1])=0. \] since the path is transversal. 
Fix an $\vep>0$ small enough, for any $T>0$ large enough,  from the homotopy property of Maslov index, we have
\[\mu(V_d,\hat{\gamma}(\tau,-T)e^{\vep J}V_d; [-T,T])= \mu(V_d,\hat{\gamma}(\tau,-T)V_d; [-T,T])-2. \] 
Together with \eqref{4.8}, it shows 
\begin{equation}
\label{eq: m mas} m^-(x;-T,T)=\mu(V_d,\hat{\gamma}(\tau,-T)e^{\vep J}V_d; [-T,T]). 
\end{equation}

Now we will try to estimate  $\mu(V_d,\hat{\gamma}(\tau,-T)e^{\vep J}V_d; [-T,T])-\mu(V_d,V^u(\tau); [-T,T])$.  

For this, let $\Lambda_s$, $s\in[0,1]$ be a path of Lagrangian subspaces of $\rr^4$ with  $\Lambda_0=V^-(-T)$ and $ \Lambda(1)=e^{\vep J}V_d$, then by the homotopy invariant property of Maslov index, 
\begin{multline}
\label{eq: comd1}  \mu(V_d,\Lambda_s;[0,1])+\mu(V_d,\hat{\gamma}(\tau,-T)e^{\vep J}V_d; [-T,T])= \\ \mu(V_d, V^-(\tau); [-T,T])+\mu(V_d,\gamma(T,-T)\Lambda_s;[0,1] ).
\end{multline}
Then we have, 
\begin{multline}
\label{eq: comd2}  \mu(V_d,\hat{\gamma}(\tau,-T)e^{\varepsilon J}V_d; [-T,T])-\mu(V_d,V^-(\tau); [-T,T]) \\ = \mu(V_d,\gamma(T,-T)\Lambda_s;[0,1] ) - \mu(V_d,\Lambda_s;[0,1]) =s(\gamma(T,-T)^{-1}V_d,V_d; V^-(-T),e^{\varepsilon J}V_d),
\end{multline}
where $s(.,.;.,.)$ is the H\"{o}rmander index (see \eqref{hp} in Appendix). As $V_d \pitchfork V^-(T)$, 
\begin{equation}
\label{eq: comd3}  \lim_{T\to +\infty}\gamma(T,-T)^{-1}V_d=V^-(J\bh_-). 
 \end{equation}
Since the above hold for any $T$ large enough, we get 
\begin{equation*}
\label{eq: comd4}  \lim_{T\to +\infty }s(\gamma(T,-T)^{-1}V_d,V_d; V^-(-T),e^{\vep J}V_d)=s(V^-(J\bh_-),V_d; V^-(-\infty),e^{\vep J}V_d). \end{equation*}
Recall that 
$$ V^-(J\bh_-)=\hat{e}^1_{-}(\psi_0^-, \tht_0^-) \wedge \hat{e}^2_{-}(\psi_0^-, \tht_0^-), \;\; V^-(-\infty)=\hat{e}^1_{+}(\psi_0^-, \tht_0^-)\wedge \hat{e}^2_{+}(\psi_0^-, \tht_0^-). $$
Let $V_d= V_d^1 \oplus V_d^2$, where $V_d^1= \rr \oplus 0$ and $V_d^2 = \rr \oplus 0$, then
 \begin{multline*} 
 s(V^-(J\bh_-),V_d; V^-(-\infty),e^{\varepsilon J}V_d)= s( \langle \hat{e}^1_{-}(\psi_0^-, \tht_0^-) \rangle, V_d^1; \langle \hat{e}^1_{+}(\psi_0^-, \tht_0^-) \rangle, e^{\varepsilon J}V_d^1) \\ +s(\langle \hat{e}^2_{-}(\psi_0^-, \tht_0^-) \rangle, V_d^2; \langle \hat{e}^2_{+}(\psi_0^-, \tht_0^-) \rangle,e^{\vep J}V_d^2). 
 \end{multline*} 
A simple computation shows 
$$s( \langle \hat{e}^1_{-}(\psi_0^-, \tht_0^-) \rangle, V_d^1; \langle \hat{e}^1_{+}(\psi_0^-, \tht_0^-) \rangle, e^{\varepsilon J}V_d^1)=0; $$
$$s(\langle \hat{e}^2_{-}(\psi_0^-, \tht_0^-) \rangle, V_d^2; \langle \hat{e}^2_{+}(\psi_0^-, \tht_0^-) \rangle,e^{\vep J}V_d^2)=0. $$
This means $m^-(x) = \mu(x)$. 
\end{proof}

While the above theorem connects $m^-(x)$ with $\mu(x)$, the next one will does the same for $i(x)$ and $\mu(V_d,V(\tau); \mathbb R)$, where
$$ \mu(V_d,V(\tau); \mathbb R) = \lim_{T \to +\infty} \mu(V_d, V(\tau); [-T, T]). $$
The limit exists, when $\Delta(\tht_0^{\pm})>0$, as it implies $V(\pm \infty) \pitchfork V_d$, see Proposition \ref{prop: lim V}. 

\begin{thm}\label{thm4.2} When $\Delta(\tht_0^{\pm}) >0$, $ i(x)=\mu(V_d,V(\tau); \mathbb R). $
\end{thm}
\begin{proof}    From \eqref{1.3c.1}, we have 
\begin{equation} 
\mu(V_d,V(\tau); R)=\sum_{\tau \in \rr} \dim(V_d \cap V(\tau)),   \end{equation} 
Recall that $V(\tau)=\text{span}\{\eta_1(\tau),\eta_2(\tau) \}$, where $\eta_1(\tau),\eta_2(\tau)$ are defined in \eqref{eq: eta 1} and \eqref{eq: eta 2}. Obviously $\eta_2(\tau) \notin V_d $,  $ \forall \tau \in \mathbb R$, which implies $\dim(V(\tau)\cap V_d) \leq 1$. 

We claim $\dim (V(\tau)\cap V_d)=1 $, if and only if $u(\tau)=0$. Assume there is a non-zero $\eta(\tau)=\beta_1 \eta_1(\tau)+ \beta_2 \eta_2(\tau)$, which is also contained in $V_d$. Then it must satisfies the following two equations
\begin{equation}  
\beta_1v-\frac{2\beta_2}{2+\alpha}=0; \quad \beta_1 u=0. 
\end{equation} 
However the above equations has a solution if and only if $u=0$. 

Meanwhile by the fourth equation in \eqref{eq: McGehee 2}, 
$$ u(\tau) = \tht'(\tau) = \dot{\tht}(t(\tau)) r^{-(1 + \frac{\al}{2})}(t(\tau)). $$
Since $r(\tau) > 0$, for any $\tau \in \rr$, we have $\frac{dt}{d\tau} = r^{1+\frac{\al}{2}}>0$ and 
\begin{equation}
\label{eq: u tau =0} i(x)=  \# \{t \in (T^-, T^+)| \; \dot{\theta}(t)=0\} =\# \{\tau \in \rr| \; u(\tau) =0 \}.
\end{equation}
This finishes our proof. 

\end{proof}

With the above result, we just need to estimate the difference between $\mu(x)$ and $\mu(V_d, V(\tau); \rr)$, which is exactly the purpose of our next lemma.

\begin{lem}\label{lem4.2} 
Assume $\pi(z)(\tau)$ is type-I or III and $\Delta(\tht_0^{\pm})>0$, when $\tau \to -\infty$, 
\begin{enumerate}
\item[(a).] if $\pi(z)(\tau) \to (\psi_0^-, \tht_0^-)$ along $\langle e_+(\psi_0^-, \tht_0^-) \rangle$, then $\mu(x) - \mu(V_d, V(\tau); \rr)=0$;
\item[(b).] if $\pi(z)(\tau) \to (\psi_0^-, \tht_0^-)$ along $\langle e_-(\psi_0^-, \tht_0^-) \rangle$, then $\mu(x) - \mu(V_d, V(\tau); \rr)=0 \text{ or } 1.$
\end{enumerate}
\end{lem}
\begin{proof} 

Notice that when $\pi(z)(\tau)$ is type-I or III, $\psi_0^- = -\pi/2$. If $\pi(z)(\tau) \to (\psi_0^-, \tht_0^-)$ along $\langle e_+(\psi_0^-, \tht_0^-) \rangle$, as $\tau \to -\infty$, by Proposition \ref{prop: lim V}, 
$$\lim_{\tau \to -\infty}V(\tau) =\text{span}\{\eh^1_{+}(\psi_0^-, \tht_0^-), \eh^2_{+}(\psi_0^-, \tht_0^-) \} = V^+(J\bh_-).$$
Then Lemma \ref{lem: Vs Vu} implies $V(\tau) = V^-(\tau)$, $\forall \tau \in \rr$, so $\mu(x) = \mu(V_d, V(\tau); \rr)$. This proves property (a). 

Now assume $\pi(z)(\tau) \to (\psi_0^-, \tht_0^-)$ along $\langle e_-(\psi_0^-, \tht_0^-) \rangle$, as $\tau \to -\infty$. Fix an arbitrary $T>0$ large enough in the following. It will be enough for us to prove 
\begin{equation}
 \label{eq: Vu - V} \mu(V_d, V^u(\tau); [-T, T]) - \mu(V_d, V(\tau); [-T, T]) = 0 \text{ or } 1.
 \end{equation} 
By the proof of Theorem \ref{thm4.1}, for a given $\vep>0$ small enough, 
$$ \mu(V_d, V^u(\tau); [-T, T]) = \mu(V_d, \gh(\tau, -T)e^{\vep J}V_d; [-T, T]).$$
Hence instead of \eqref{eq: Vu - V}, we will show the following 
\begin{equation*}
\label{eq: gmV - V} \mu(V_d, \gh(\tau, -T)e^{\vep J}V_d; [-T, T])-\mu(V_d, V(\tau); [-T, T]) = 0 \text{ or } 1.
\end{equation*}

Let $\Lambda_s$, $s\in[0,1]$ be a path of Lagrangian subspaces of $\rr^4$ with  $\Lambda_0=V(-T)$ and $ \Lambda_1=e^{\vep J}V_d$. Similar to \eqref{eq: comd2}, we have 
\begin{multline*}
\label{eq: comd2.1}  \mu(V_d,\hat{\gamma}(\tau,-T)e^{\vep J}V_d; -T,T)-\mu(V_d,V(\tau); -T,T) \\ =s(\gamma(T,-T)^{-1}V_d,V_d; V(-T),e^{\vep J}V_d).
\end{multline*}
Notice that $\lim_{T \to +\infty}\gamma(T,-T)^{-1}V_d=V^-(J\hat{B}_-)=\hat{e}^1_{-}(\psi_0^-, \tht_0^-)\wedge \hat{e}^2_{-}(\psi_0^-, \tht_0^-) $, and under the condition of property (b),  $\lim_{T\to +\infty}V(-T)=V(-\infty)=\hat{e}^1_{+}(\psi_0^-, \tht_0^-)\wedge \hat{e}^2_{-}(\psi_0^-, \tht_0^-) $. Therefore 
\[ s(\gamma(T,-T)^{-1}V_d,V_d; V(-T),e^{\vep J}V_d)=s(\gamma(T,-T)^{-1}V_d, e^{\frac{\vep}{2} J}V_d; V(-T),e^{\vep J}V_d). \]
When $\psi_0^-=-\pi/2$, by Lemma \ref{lem: B hat 1}
$$ \eh^1_{-}(\psi_0^-, \tht_0^-)=\left( -(\al+1)\sqrt{2\uh(\tht_0^-)}, 1 \right)^T, \;\; \eh^1_{+}(\psi_0^-, \tht_0^-)= \left( \frac{\al}{2} \sqrt{2\uh(\tht_0^-)}, 1 \right)^T,  $$
$$ \eh^2_{-}(\psi_0^-, \tht_0^-)= \left(\frac{2-\al}{4}\sqrt{2\uh(\tht_0^{-})}-\ey \sqrt{\Delta(\tht_0^{-})}, 1 \right)^T.$$
For simplicity, set $b=\sqrt{2\uh(\tht_0^-)}, c=\frac{2-\al}{4}\sqrt{2\uh(\tht_0^{-})}-\ey \sqrt{\Delta(\tht_0^{-})} $. 

Write the Lagrangian subspaces as graphs of linear maps: $ 0 \oplus \rr^2\to V_d$:
$$ V^-(J\hat{B}_-)= \text{Gr}(A_0), e^{\frac{\varepsilon}{2} J}V_d = \text{Gr}(A_1); \;\; V(-\infty) = \text{Gr}(B_0), e^{\varepsilon J}V_d= \text{Gr}(B_1),$$ 
where
  \[ A_0=\left(\begin{array}{cc}-(\alpha+1)b &0\\ 
                             0 & c \end{array}\right), \quad  A_1=\left(\begin{array}{cc}\cot(\vep/2) &0\\ 
                             0 & \cot(\vep/2) \end{array}\right);\] 
 \[ B_0=\left(\begin{array}{cc}\frac{\alpha}{2}b &0\\ 
                             0 & c \end{array}\right), \quad  B_1=\left(\begin{array}{cc}\cot(\vep) &0\\ 
                             0 & \cot(\vep) \end{array}\right).\]
  
Let $A_{0,T}, B_{0,T}$ be the matrices, such that  $\gamma(T,-T)^{-1}V_d=\text{Gr}(A_{0,T})$ and $V(-T)=\text{Gr}(B_{0,T})$. Then for $T$ large enough, $A_{0,T}, B_{0,T}$ are in the $\vep/2$-neighborhood of $A_0, B_0$ correspondingly. By the property of H\"{o}rmander index (see \eqref{hc}),  
 
\begin{multline*}
s(\gamma(T,-T)^{-1}V_d, e^{\frac{\vep}{2} J}V_d; V(-T),e^{\vep J}V_d)) = \frac{1}{2}\text{sign}(B_{0,T}-A_1)+ \\ \frac{1}{2}\text{sign}(B_1-A_{0,T}) -\frac{1}{2}\text{sign}(B_1-A_1)-\frac{1}{2}\text{sign}(B_{0,T}-A_{0,T}).
\end{multline*}

Notice that $B_{0,T}-A_1, B_1-A_1$ are negative definite, and $B_1-A_{0,T}$ is positive definite. Hence
\[ -\frac{1}{2}\text{sign}(B_{0,T}-A_1)=\frac{1}{2}\text{sign}(B_1-A_{0,T}) 
=  -\frac{1}{2}\text{sign}(B_1-A_1)=1. \]    
Since $B_{0,T}-A_{0,T}$ is in the $\vep$-neighborhood of $B_0-A_0$, which has a positive eigenvalue $ (1+\frac{3}{2}\alpha) $, we have  
\[\frac{1}{2}\text{sign}(B_{0,T}-A_{0,T})=0,\quad \text{or} \quad 1.\]
This completes our proof.
\end{proof}

\begin{cor}
\label{cor: hyperbolic m i} Assume $\Delta(\tht_0^{\pm})>0$. When $\pi(z)(\tau)$ is a type-I or III heteroclinic orbit,  
\begin{enumerate}
\item[(a).] if $\pi(z)(\tau) \to (\psi_0^-, \tht_0^-)$ along $\langle e_+(\psi_0^-, \tht_0^-)  \rangle$, as $\tau \to -\infty$, then  $m^-(x)=i(x)$;
\item[(b).] if $\pi(z)(\tau) \to (\psi_0^-, \tht_0^-)$ along $\langle e_-(\psi_0^-, \tht_0^-)  \rangle$, as $\tau \to -\infty$, then  $m^-(x) - i(x)= 0 \text{ or } 1.$
\end{enumerate}
When $\pi(z)(\tau)$ is a type-II heteroclinic orbit, 
\begin{enumerate}
\item[(c).] if $\pi(z)(\tau) \to (\psi^+_0, \tht_0^+)$ along $\langle e_-(\psi_0^+, \tht_0^+) \rangle$, as $\tau \to +\infty$, then $m^-(x)= i(x)$; 
\item[(d).] if $\pi(z)(\tau) \to (\psi^+_0, \tht_0^+)$ along $\langle e_+(\psi_0^+, \tht_0^+)  \rangle$, as $\tau \to +\infty$, then $m^-(x) - i(x)= 0 \text{ or } 1.$
\end{enumerate}
\end{cor}

\begin{proof}
Property (a) and (b) follows directly from Theorem \ref{thm4.1}, \ref{thm4.2} and Lemma \ref{lem4.2}. 

For property (c) and (d), as the corresponding $x(t)$ is a collision-parabolic solution, $\tilde{x}(t) = x(-t)$ will be a parabolic-collision solution. By their definitions, it is not hard to see $m^-(\tilde{x}) = m^-(x)$ and $i(\tilde{x})= i(x)$. 

Let $\tilde{z}$ be the zero energy orbit of \eqref{eq: Hamiltonian equation} corresponding to $\tilde{x}$, and $(\tilde{v}, \tilde{u}, \tilde{r}, \tilde{\tht})(\tau)$ the corresponding orbit in McGehee coordinates, then by the computation given at the beginning of Section \ref{sec: McGehee coordinates}, we have
$$ (\tilde{v}, \tilde{u}, \tilde{r}, \tilde{\tht})(\tau)= (-v, -u, r, \tht)(-\tau). $$  
As a result, on the collision manifold $\mf$ with coordinates defined in \eqref{eq: uv psi}, we have 
$$ (\tilde{\psi}, \tilde{\tht})(\tau) = (\psi+\pi, \tht)(-\tau). $$
Then
$$ (\tilde{\psi}^-_0, \tilde{\tht}^-_0):= \lim_{\tau \to -\infty}(\tilde{\psi}, \tilde{\tht})(\tau) = \lim_{\tau \to -\infty} (\psi+\pi, \tht)(-\tau) = (\psi_0^+ +\pi, \tht_0^+). $$
By \eqref{eq: linear equil}, 
$$ M(\tilde{\psi}_0^-, \tilde{\tht}_0^-) = M(\psi_0^+ +\pi, \tht_0^+)= -M(\psi_0^+, \tht_0^+). $$ 
As a result,
$$ e_+(\tilde{\psi}_0^-, \tilde{\tht}_0^-) = e_-(\psi_0^+, \tht_0^+), \;\; e_-(\tilde{\psi}_0^-, \tilde{\tht}_0^-) = e_+(\psi_0^+, \tht_0^+). $$
Then the rest follows from property (a) and (b), which we have already proven. 
\end{proof}

In the above we always assume $\Delta(\tht_0^{\pm})>0$, to deal the non-hyperbolic case, i.e., $\Delta(\tht_0^{\pm}) <0$, we have the next proposition
\begin{prop}
\label{prop: non-hyper index} If at least one of $\Delta(\tht_0^{\pm})$ is negative, then $m^-(x) = i(x)=+\infty$.
\end{prop}
 
\begin{proof}
We only give the details for the case $\Delta(\tht_0^{+}) <0$, while the proof for the other case is exactly the same. 

For $\vep>0$ small enough, we can find a $\tau_0>0$, such that $\|\hat{B}(\tau)-\hat{B}_+\|<\vep$, for any $\tau\in[\tau_0,+\infty)$. By \eqref{4.8}, if $ \lim_{\tau_1 \to +\infty} \mu(V_d,\hat{\gamma}(\tau,\tau_0))V_d; [\tau_0, \tau_1])= +\infty$, then $m^-(x)=+\infty$.
Since $\hat{B}(\tau)>\hat{B}_+-\vep I_4$, from the monotonic property of Maslov index, 
\begin{equation*}
 \mu(V_d,\hat{\gamma}(\tau,\tau_0))V_d; [\tau_0, \tau_1])\geq  \mu(V_d,e^{(\tau-\tau_0)(\hat{B}_+-\vep I_4)}V_d; [\tau_0, \tau_1]), \; \forall \tau_1>\tau_0. 
 \end{equation*} 
By the symplectic additivity property,
\begin{multline*}
\mu(V_d,e^{(\tau-\tau_0)(\hat{B}_+-\vep I_4)}V_d)= \\ \mu(V_d,e^{(\tau-\tau_0)(\hat{B}^{(1)}_+-\vep I_2)}V_d)+\mu(V_d,e^{(\tau-\tau_0)(\hat{B}^{(2)}_+-\vep I_2)}V_d).
\end{multline*}
Since in this case the crossing form is always positive, $\mu(V_d,\exp((\tau-\tau_0)(\hat{B}^{(i)}_+-\vep I_2))V_d)$ is the summation of  $\dim(\exp((\tau-\tau_0)(\hat{B}^{(i)}_+-\vep I_2)V_d)\cap V_d)$ over $\tau\in[\tau_0,\tau_1]$, for $i =1, 2$. As a result, 
\[ \mu(V_d,e^{(\tau-\tau_0)(\hat{B}_+-\vep I_4)}V_d)\geq \sum_{\tau\in[\tau_0,\tau_1]}\dim e^{(\tau-\tau_0)(\hat{B}^{(2)}_+-\vep I_2)}V_d \cap V_d.  \]
Notice that 
$$ \hat{B}^{(2)}_+-\vep I_2=\left(
 \begin{array}{cccc}
 1-\vep & \frac{2-\alpha}{4}\sqrt{2\uh(\theta_0^+)} \\
\frac{2-\alpha}{4}\sqrt{2\uh(\theta_0^+)} & -\uh_{\tht \tht}(\theta_0^+)-\vep  
\end{array}
\right).$$
For $\vep$ small enough, $ \hat{B}^{(2)}_+-\vep I_2>0$, a direct computation shows the summation of crossing time is unbounded as $\tau_1 \to +\infty$. Hence $m^-(x)=+\infty$. 

For $i(x)$, when $\Delta(\tht_0^+) < 0$, by Lemma \ref{lem: equil eigenvalue} and \ref{lem: equi locally}, $(\psi_0^+, \tht_0^+)$ is a stable focus. As a heteroclinic orbit on $\mf$, $\pi(z)(\tau)$ spiral into $(\psi_0^+, \tht_0^1)$ as $\tau \to +\infty$. Therefore $i(x) = \# \{ \tau \in \rr: \tht'(\tau) =0 \}= +\infty$.
\end{proof}

Now we are ready to prove our main results. 

\begin{proof}[Proof of Theorem \ref{thm: osc morse}]
The fact the $\tht_0^{\pm}$ are critical points of $\uh$ follows directly from Theorem \ref{thm: zero energy}. Let's assume $x(t)$ is either a parabolic or a parabolic-collision solution (a collision-parabolic solution becomes parabolic-collision after reversing time). Since $x(t)$ is not homothetic, the corresponding orbit $\pi(z)(\tau)$ on the collision manifold $\mf$ is either a type-I or III heteroclinic orbit.  

Now if one of $\Delta(\tht_0^{\pm})$ is negative, then $m^-(x)= i(x)= +\infty$, by Proposition \ref{prop: non-hyper index}. This proves property (a). If both $\Delta(\tht_0^{\pm})$ are positive, then property (b) follows from Lemma \ref{lem: heteroclinic types}. In particular, when $\tht_0^-$ is a non-degenerate local minimizer, $\uh_{\tht \tht}(\tht_0^-)>0$. Then by Lemma \ref{lem: equil eigenvalue}, $\lmd_-(\psi_0^-, \tht_0^-) < 0 < \lmd_{+}(\psi_0^-, \tht_0^+)$, and by Lemma \ref{lem: equi locally}, the unstable manifold of $(\psi_0^-, \tht_0^-)$ in the collision manifold is tangent to $\langle e_+(\psi_0^-, \tht_0^-) \rangle$. Hence $\pi(z)(\tau)$ approaches to $(\psi_0^-, \tht_0^-)$ along $\langle e_+(\psi_0^-, \tht_0^-) \rangle$, as $\tau \to -\infty$. Then by property (a) in Corollary \ref{cor: hyperbolic m i}, $m^-(x)=i(x)$. 
\end{proof}

\begin{proof}[Proof of Corollary \ref{cor: saddle to saddle}] 
By Theorem \ref{thm: osc morse}, it is enough to show $i(x)=0$. Meanwhile by \eqref{eq: u tau =0} and \eqref{eq: uv psi}, this is equivalent to $\psi(\tau) \ne \pm \pi/2$, for any $\tau \in \rr$.

Assume $\psi(\tau_0)= \frac{\pi}{2}$ (the case for $\psi(\tau_0)=-\frac{\pi}{2}$ is similar), for some $\tau_0 \in \rr$, then $v(\tau_0)= \sqrt{\uh(\tht(\tau_0))}$. Recall that $v(\tau)$ is a non-decreasing function of $\tau$, so $v(\tau_0)=\sqrt{\uh(\tht(\tau_0))} \le v(+\infty)= \sqrt{\uh(\tht_0^{+})}$. This means $\tht(\tau_0)$ must be a global minimizer of $\uh$ as well. Then by Lemma \ref{lem: equail}, $(\psi(\tau_0), \tht(\tau_0))$ is a equilibrium in the collision manifold, which is absurd. 
\end{proof}

Our next proof follows ideas from \cite{BS} and \cite{BHPT}. 

\begin{proof} [Proof of Theorem \ref{thm: homothtic}]
Without loss of generality let's assume $\xb(t)$ is a collision-parabolic solution defined on $\rr^+=(0, +\infty)$. With the energy being zero, we have 
$$ \rb(t) = (\kp t)^{\frac{2}{2+\al}}, \text{ where } \kp = \frac{2+\al}{2} \sqrt{2\uh(\tht_0)}.$$

Recall that in polar coordinates, the action functional is 
$$ \cf(r, \tht)= \int \ey \dot{r}^2 + \ey r^2 \dot{\tht}^2 +r^{-\al} \uh(\tht) \,dt. $$
By results from  \cite{LM14}, for any $ [t_0, t_1] \subset (0, +\infty)$,  $\bar{x}(t)$ is a minimizer of $\cf$ in 
$$ \{ (r, \tht) \in W^{1,2}([t_0, t_1], \rr^+ \times \mathbb{S}^1): \; r(t_0)= \rb(t_0), r(t_1) = \rb(t_1), \tht(t) \equiv \tht_0 \}. $$ 
Therefore we only need to consider variations of $\cf$ along $\phi \in C_0^{\infty}(\rr^+, \mathbb{S}^1)$ (smooth functions with compact supports). The second derivative of $\cf$ along such a $\phi$ is 
\begin{equation}
\label{eq: F sec dev tau} d^2\cf(r, \tht)[\phi, \phi] = \int r^2 \dot{\phi}^2 + r^{-\al}\uh_{\tht \tht}(\tht) \phi^2 \,dt = \int \rb^{\frac{2-\al}{2}} \big( (\phi')^2+\uh_{\tht \tht}(\tht_0) \phi^2 \big) \,d\tau,
\end{equation}

If $\xi(t)= \rb(t)^{\frac{2-\al}{4}}\phi(t)$, then $\xi' =\rb^{\frac{2-\al}{4}}\phi' + \frac{2-\al}{4}\rb^{-\frac{2+\al}{4}}r' \phi$, and
\begin{equation}
\label{eq: xi} \begin{split}
\rb^{\frac{2-\al}{2}}(\phi')^2 & = (\xi')^2 + \frac{(2-\al)^2}{16} \big(\frac{\rb'}{\rb}\big)^2 \xi^2 - \frac{2-\al}{2} \big(\frac{\rb'}{\rb} \big)\xi \xi' \\
& = (\xi')^2 + \frac{(2-\al)^2}{8}\uh(\tht_0) \xi^2 - \frac{2-\al}{2}\sqrt{2\uh(\tht_0)} \xi \xi' ,
\end{split}
\end{equation}
where the second equality following from 
$$ \frac{\rb'}{\rb} = \dot{\rb} \rb^{-1} \frac{dt}{d\tau}= \dot{\rb} \rb^{\frac{\al}{2}} = \frac{2}{2+\al} \kp = \sqrt{2 \uh(\tht_0)}. $$
Plug \eqref{eq: xi} into \eqref{eq: F sec dev tau}, we get 
$$ d^2\cf(\rb, \tht_0)[\phi, \phi]= \int (\xi')^2 + \frac{1}{4}\Delta(\tht_0) \xi^2 -\frac{2-\al}{2} \sqrt{2\uh(\tht_0)} \xi \xi'\,d\tau.  $$
As $\xi$ has a compact support in $\rr^+$, using integration by parts,
$$ d^2\cf(\rb, \tht_0)[\phi, \phi]= \int (\xi')^2 +\frac{1}{4}\Delta(\tht_0) \xi^2 \,d\tau.  $$

When $\Delta(\tht_0) \ge 0$, $d^2 \cf(\rb, \tht_0)[\phi, \phi] \ge 0$, for any $\phi$, so $m^-(x)=0$. When $\Delta(\tht_0)< 0$, there is a countable set of linear independent functions $\{\phi_n \in C^{\infty}_0(\rr^+, \mathbb{S}^1): n \in \mathbb{Z}^+ \}$ satisfying $d^2 \cf(\rb, \tht_0)[\phi_n, \phi_n]<0$ (see \cite[Theorem 4.3]{BS}), which implies $m^-(x) = +\infty$. 
\end{proof}



\section{Application in Celestial Mechanics} \label{sec: application}

In this section, we give some applications of our results to celestial mechanics. 

\subsection{The planar isosceles three body problem } \label{sec: isosceles}
Consider the problem of three point masses, $m_i$, $i=1,2,3$, in a plane moving under the Newtonian gravitational force of each other. Let $q=(q_1, q_2, q_3)$, where $q_i$ represents the position of $m_i$, and $p = M \dot{q}$, where $M= \text{diag}(m_1, m_1, m_2, m_2, m_3, m_3)$, then 
\begin{equation} \label{eq: three body}
\dot{p} = \nabla_q \tilde{U}(q); \;\; \dot{q} = M^{-1} p, 
\end{equation}
where $\tilde{U}(q)= \sum_{1 \le i < j \le 3} \frac{m_i m_j}{|q_i -q_j|}$, is the (negative) potential. This is equivalent to the Euler-Lagrangian equation $\frac{d}{dt}L_{\dot{q}}(q, \dot{q})= L_q(q, \dot{q})$ with
\begin{equation}
\label{eq: Lagrangian three body} L(q, \dot{q}) = K(\dot{q}) +\tilde{U}(q) = \ey |\dot{q}|_M^2 + \tilde{U}(q), 
\end{equation}
where $|w|_M:=(\sum_{i=1}^3 m_i |w_i|^2)^{\ey}$, for any $w=(w_1, w_2, w_3) \in \rr^{2 \times 3}$. 

The above problem has six degrees of freedom. It can be reduced to four after fixing the center of mass at the origin, $\sum_{i=1}^3 m_i q_i=0$. Moreover when two of the masses are equal ($m_1 =m_2$), it has an invariant sub-system with two degrees of freedom, where the three masses form an isosceles triangle all the time: 
\begin{equation}
\label{eq: reflection} \{q=(q_1, q_2, q_3)| \; q_2 = \mr(q_1), q_3 = \mr(q_3) \}.
\end{equation}
Here $\mr$ represents the reflection in $\rr^2$ with respect to the vertical axis. 

For simplicity, we assume $m_1= m_2 =m$ and $m_3=1$. Let
$$ r= |q|_M  \; \text{ and } \; s_i = q_i/r, \;\; \forall i =1,2,3. $$
Then $s=(s_1, s_2, s_3)$ satisfies  $|s|_M =1$. Set $s_1=(\xi, \eta)$, by \eqref{eq: reflection}, 
\begin{equation*}
s_2= (-\xi, \eta), \;\; s_3 = (0, -2m \eta),
\end{equation*}
which means 
$$ |s|^2_M=2m \xi^2 + 2m(2m+1) \eta^2=1. $$
This allows us to introduce an angular variable $\tht \in \mathbb{S}^1$ by
$$ \xi= \frac{\cos \tht}{\sqrt{2m}}, \;\; \eta= \frac{\sin \tht}{\sqrt{2m (2m+1)}}. $$

Under the new variables $(r, \tht)$, the Lagrangian of the isosceles three problem has the following expression which fits the framework of this paper:
\begin{equation*}
\label{eq: uh isosceles} L(r, \tht, \dot{r}, \dot{\tht}) = \ey (\dot{r}^2 +r^2 \dot{\tht}^2) +\frac{\uh(\tht)}{r}, \; \text{where } \uh(\tht) = \frac{m^{\frac{5}{2}}}{\sqrt{2}|\cos \tht|} + \frac{2 \sqrt{2}m^{\frac{3}{2}}}{(1 +2m \sin^2 \tht)^{\frac{1}{2}}}.
\end{equation*}
However besides the singularity at the origin, $r=0$, corresponding a triple collision. There are additional singularities at $\tht= \pm \frac{\pi}{2}$ due to binary collisions between $m_1$ and $m_2$. Although a double collision can be regularized (see \cite[Section 7]{SM95} or \cite{Mk81}), it is not so clear how to define the corresponding Morse index in this case, so when applying our results, we have to restrict ourselves to a domain of the zero energy solution, where there is no binary collision. 

It is easy to see $\uh(\tht)$ has four different non-degenerate global minima: 
$$ -\pi+\tht^*< -\tht^*< \tht<  \pi-\tht^*, \; \text{ for some } \tht^* \in (0, \pi/2),$$
which are the Lagrangian configurations, where the three masses form an equilateral triangle. The second derivatives of $\uh(\tht)$ at these critical points all are positive, so the condition required in Lemma \ref{lem: B hat 1} always holds at these points.

Meanwhile there are two non-degenerate critical points at $\tht = 0$ or $\pi$, which are local maxima of $\uh$. They are the Euler configurations with $m_3$ at the origin. By a direct computation, 
$$ \uh(0) = \uh(\pi)= \frac{m^{\frac{5}{2}}}{\sqrt{2}}+ 2 \sqrt{2} m^{\frac{3}{2}}, \;\; \uh_{\tht \tht}(0) = \uh_{\tht \tht}(\pi) = -\frac{7}{\sqrt{2}}m^{\frac{5}{2}}. $$
Recall that for $\al=1$, $\Delta(\tht_0) = \ey \uh(\tht_0) +4 \uh_{\tht \tht}(\tht_0)$. Then $\Delta(0)= \Delta(\pi)$ are positive, when $m < 4/55$, and negative, when $m > 4/55.$ As shown by Moeckel \cite{Mk81}, if a zero energy solution (non-homothetic) approaches to the origin or the infinity along the horizontal axis (or equivalently the configuration formed by the three masses converges to a Euler configuration), then for a generic $m > 4/55$, during the process, the three masses oscillate frequently along the horizontal axis. This corresponds to the change of the sign of $\dot{\tht}(t)$, which by our results gives an estimate of the Morse index of the solution.

\subsection{The Kepler-type problem} \label{sec: Kepler}
In our results, we require the critical points of $\uh$ to be non-degenerate. In general our approach may still work even when this condition does not hold. What we need is the knowledge of the asymptotic behavior of $V(\tau)$ defined in Lemma \ref{lem: V Lag}, as $\tau$ goes to infinity. This is important as in celestial mechanics these critical points corresponds to central configurations, which are degenerate due to symmetries. As an example, we will consider the \emph{Kepler-type} problem, where each $\tht$ is a degenerate critical point of $\uh$:
$$ \uh(\tht) \equiv m, \; \; \forall \tht \in \mathbb{S}^1, \text{ for some constant } m>0. $$ 

Now the vector field \eqref{eq: vector field collision mfd} on the collision manifold $\mf$ becomes
\begin{equation} \label{eq: vector field kepler}
\begin{cases}
\psi'&=(1-\frac{\alpha}{2}) \sqrt{2m} \cos\psi, \\
  \theta'&= \sqrt{2m} \cos\psi,
\end{cases} 
\end{equation}
and very $(\psi_0, \tht_0)$ with $\psi_0 \in \{ \pm \frac{\pi}{2}\}$, $\tht_0 \in \mathbb{S}^1$, is an equilibrium. Let $M(\psi_0, \tht_0)$ be defined as in \eqref{eq: linear equil}. Following Notation \ref{dfn: eigen value vector}, by Lemma \ref{lem: equil eigenvalue}, 
\begin{equation*}
\begin{cases}
& -(2 -\al)= \lmd_-(\frac{\pi}{2}, \tht_0) <  \lmd_+(\frac{\pi}{2}, \tht_0) =0; \\
& e_-(\frac{\pi}{2}, \tht_0) = (\frac{2-\al}{2}, 1)^T, \;\; e_+(\frac{\pi}{2}, \tht_0) = (0, 1)^T, 
\end{cases}
\end{equation*}
\begin{equation*}
\begin{cases}
& 0= \lmd_-(-\frac{\pi}{2}, \tht_0) <  \lmd_+(-\frac{\pi}{2}, \tht_0) =(2 -\al); \\
& e_-(-\frac{\pi}{2}, \tht_0) =  (0, 1)^T, \;\; e_+(-\frac{\pi}{2}, \tht_0) = (\frac{2-\al}{2}, 1)^T
\end{cases}
\end{equation*}

Let $x(t)$ be a parabolic solution of the Kepler-type problem, then its projection to the collision manifold is a heteroclinic orbit going from $(-\frac{\pi}{2}, \tht_0^-)$ to $(\frac{\pi}{2}, \tht_0^+)$. Since $\tht_0^{\pm}$ are degenerate, Lemma \ref{lem: equi locally} does not apply. However by \eqref{eq: vector field kepler},
\begin{equation}
\label{eq: lim psi' tht'} \frac{\psi'}{\tht'}(\tau) = \frac{2-\al}{2}, \;\; \forall \tau \in \rr.
\end{equation}
Hence the heteroclinic orbit converges to $(\frac{\pi}{2}, \tht_0^+)$ along the subspace $\langle e_-(\frac{\pi}{2}, \tht_0^+) \rangle$, as $\tau \to +\infty$, and converges to $(-\frac{\pi}{2}, \tht_0^-)$ along the subspace $\langle e_+(-\frac{\pi}{2}, \tht_0^- \rangle$, as $\tau \to -\infty$, which means it is a type-I heteroclinic orbit. 

Let $V(\tau) =\text{span}\{\eta_1(\tau), \eta_2(\tau) \}$ be the path of Lagrangian subspaces given in Definition \ref{defi: V Lag}. With \eqref{eq: lim psi' tht'}, the same computation used in the proof of Lemma \ref{lem: lim uh_tht u} shows $\lim_{\tau \to \pm \infty} \frac{\uh_{\tht}(\tht)}{u} = 0.$ Recall that $\lmd_+(\frac{\pi}{2}, \tht_0^+)= \lmd_-(\frac{\pi}{2}, \tht_0^-)=0$, so results of Lemma \ref{lem: lim uh_tht u} still hold. Then by Proposition \ref{prop: lim V}, 
$$ \lim_{\tau \to \pm\infty} V(\tau) = \text{span}\{\eh^1_{\mp}(\pm\frac{\pi}{2}, \tht_0^\pm), \eh^2_{\mp}(\pm\frac{\pi}{2}, \tht_0^\pm) \}. $$
Notice that for the Kepler-type potential,
$$ \eh^1_{-}(\pi/2, \tht_0^+)= (-\al \sqrt{m/2}, 1)^T, \;\; \eh^2_{-}(\pi/2, \tht_0^+) = (-(2-\al)\sqrt{m/2}, 1)^T; $$
$$ \eh^1_{+}(-\pi/2, \tht_0^-)= (\al \sqrt{m/2}, 1)^T, \;\; \eh^2_{+}(-\pi/2, \tht_0^+) = ((2-\al)\sqrt{m/2}, 1)^T. $$
This means the corresponding results in Section \ref{sec: Morse Maslov indices} will still hold. In particular, by Corollary \ref{cor: hyperbolic m i}, $ i(x) = \mu(x) = m^-(x). $

Since the angular momentum is a first integral of the Kepler-type problem, for a parabolic solution (so non-homothetic), $\dot{\tht}(t)$ is always positive or negative. Together with the above result they imply
\begin{cor}
For a Kepler-type problem, the Morse index of a parabolic solution is always zero. 
\end{cor}

\section{Appendix: a brief introduction to the Maslov index for heteroclinic orbits  }\label{sec: appendix B}

We start with a brief review of the  Maslov index theory from \cite{Ar,CLM, RS1}.  Let $(\mathbb{R}^{2n},\omega)$ be the standard
symplectic space, and $\text{Lag}(2n)$ the Lagrangian Grassmanian, i.e. the set of
Lagrangian subspaces of $(\mathbb{R}^{2n},\omega)$. Given two continuous paths $L_1(t),L_2(t)$, $t\in[a,b]$, in $\text{Lag}(2n)$, the Maslov index $\mu(L_1(t), L_2(t))$ is an integer invariant. There several different ways to define such an invariant. Here we use the one given in \cite{CLM}. Following are some properties of the Maslov index (for the details see \cite{CLM}).

\textbf{Property I. (Reparametrization invariance)}  Let $\vr:[c,d]\rightarrow [a,b]$ be a continuous and piecewise smooth function satisfying $\vr(c)=a$, $\vr(d)=b$, then \begin{equation} 
\mu(L_1(t), L_2(t))=\mu(L_1(\vr(\tau)), L_2(\vr(\tau))). \label{adp1.1} 
\end{equation}

\textbf{Property II. (Homotopy invariant with end points)} If two continuous
families of Lagrangian paths $L_1(s,t)$, $L_2(s,t)$, $0\leq s\leq 1$, $a\leq
t\leq b$ satisfies  $\text{dim}(L_1(s,a)\cap L_2(s,a))=C_1, \text{dim}(L_1(s,b)\cap L_2(s,b)) = C_2,$, for any $0\leq s\leq 1$, where $C_1, C_2$ are two constant integers, then
\begin{equation} 
\mu(L_1(0,t), L_2(0,t))=\mu(L_1(1,t),L_2(1,t)). \label{adp1.2} 
\end{equation}

\textbf{Property III. (Path additivity)}  If $a<c<b$, then  
\begin{equation}
\mu(L_1(t),L_2(t))=\mu(L_1(t),L_2(t); [a,c])+\mu(L_1(t),L_2(t); [c,b]).
\label{adp1.3} 
\end{equation}

\textbf{Property IV. (Symplectic invariance)} Let $\ga(t)$, $t\in[a,b]$ be a
continuous path of symplectic matrices in $\Sp(2n)$, then 
\begin{equation} 
\mu(L_1(t),L_2(t))=\mu(\ga(t)L_1(t), \ga(t)L_2(t)). \label{adp1.4} 
\end{equation}

\textbf{Property V. (Symplectic additivity)} Let $W_i$, $i=1,2$, be two symplectic spaces, if $L_i \in C([a,b], Lag(W_1))$ and $\hat{L}_i \in C([a,b], Lag(W_2))$, $i =1, 2$, then 
\begin{equation} 
\mu(L_1(t)\oplus \hat{L}_1(t),L_2(t)\oplus \hat{L}_2(t))= \mu(L_1(t),L_2(t))+
\mu(\hat{L}_1(t),\hat{L}_2(t)). 
\label{adp1.5add} 
\end{equation}

\textbf{Property VI. (Symmetry)} If $L_i \in C([a,b], \text{Lag}(2n)$, $i=1,2$, then
\begin{equation}  \mu(L_1(t), L_2(t))= \dim L_1(a)\cap L_2(a)-\dim L_1(b)\cap L_2(b)  -\mu(L_2(t),L_1(t)).    \label{adp1.5sym}       
\end{equation}

When the Hamiltonian system is given by the Legender transformation of a Sturm-Liouville system, then 
 \begin{equation}
\mu(V_d,\Lambda(t))=\text{dim}(\Lambda(a)\cap V_d)+\sum_{a<t<b} \text{dim}(\Lambda(t)\cap V_d). 
\label{1.3c.1}\end{equation}
For the detail see   \cite{RS1}, \cite{HO}.


Given a Lagrangian path $t \mapsto\Lambda(t)$, the difference of the Maslov indices of it with respect to two Lagrangian subspaces $V_0,V_1 \in \text{Lag}(2n)$, is 
given in terms of the \emph{H\"{o}rmander index} (see \cite[Theorem 3.5]{RS1})
\begin{equation}
s(V_0, V_1; \Lambda(0), \Lambda(1)) = \mu(V_0, \Lambda(t))-\mu(V_1,\Lambda(t)). 
\end{equation} 
Obviously for $\varepsilon>0$ small enough, 
\begin{equation} s(V_0, V_1; \Lambda(0), \Lambda(1)) = s(V_0, V_1; e^{-\varepsilon J}\Lambda(0),e^{-\varepsilon J}\Lambda(1)), \label{hp}
\end{equation} 

The H\"{o}rmander index is independent of the choice of the path connecting $\Lambda(0)$ and $\Lambda(1)$. Under the non-degenerate condition, i.e.,  $V_0,V_1$  are
transversal to $\Lambda(0), \Lambda(1)$ correspondingly, it has the following two basic properties
\begin{equation} \label{eq: Hormander} \begin{split}
s(V_0,V_1;\Lambda(0),\Lambda(1)) & =-s(V_1,V_0;
\Lambda(0),\Lambda(1) ), \\ 
s(\Lambda(0),\Lambda(1);V_0,V_1)& =-s(V_0,V_1; \Lambda(0),\Lambda(1)
). 
\end{split}
\end{equation}
If $V_i = Gr(A_i)$, $\Lambda(i)=Gr(B_i)$ for symmetry matrices $A_i$ and $ B_i$, $i=0,1$, then 
\begin{equation} \label{hc} 
\begin{split}
s(V_0,V_1; \Lambda(0),\Lambda(1)) &=\frac{1}{2}\text{sign}(B_0-A_1)+\frac{1}{2}\text{sign}(B_1-A_0) \\
&  -\frac{1}{2}\text{sign}(B_1-A_1)-\frac{1}{2}\text{sign}(B_0-A_0),
\end{split}
\end{equation}
where for a symmetric  matrix $A$,   $\text{sign}(A)$  is  the signature of the symmetric form 
$\langle A\cdot, \cdot\rangle$. 
 A direct corollary shows that \begin{equation} |s(V_0,V_1;
\Lambda(0),\Lambda(1))|\leq 2n.  \label{hormd}\end{equation}

\mbox{}

\textbf{Acknowledgments}. We thank the anonymous referees for their helpful comments and suggestions. The second author wishes to thank School of Mathematics at Shandong University,  Ceremade at University of Paris-Dauphine and IMCCE at Paris Observatory for their hospitalities, where this work was done when he was a visitor and a postdoc there. 

\bibliographystyle{abbrv}
\bibliography{ref-parabolic}

\end{document}